\RequirePackage{fix-cm}
\documentclass[3p,preprint]{elsarticle}
\usepackage{graphicx}
\usepackage{amsmath}
\usepackage{amsfonts}
\usepackage{csquotes}
\newtheorem{theorem}{Theorem}[section]
\newtheorem{corollary}[theorem]{Corollary}
\newtheorem{lemma}[theorem]{Lemma}
\newtheorem{proposition}[theorem]{Proposition}
\newtheorem{remark}[theorem]{Remark}
\newenvironment{proof}[1][Proof]{\noindent\textbf{#1.} }{\ \rule{0.5em}{0.5em}}

\numberwithin{equation}{section}

\newcommand{\RR}{\mathbb{R}}
\newcommand{\dd}{\mathrm{d}}
\biboptions{sort&compress}

\makeatletter
\newcommand{\daggerref}{%
  \gdef\@fnmark{\ensuremath{\dagger}}%
}
\makeatother

\begin{document}

\begin{frontmatter}

\title{Conditioning and Singular Value Distribution of Jacobi Weighted Histopolation Matrices}

\author[address-Pau]{Allal Guessab\daggerref}
\ead{allal.guessab@univ-pau.fr}

\author[address-CS]{Federico Nudo}
\ead{federico.nudo@unical.it}

\author[address-Como,address-Como2]{Stefano Serra-Capizzano}
\ead{s.serracapizzano@uninsubria.it}

\address[address-Pau]{Laboratoire de Math\'{e}matiques et de leurs Applications, UMR CNRS 5142,\\ Universit\'{e} de Pau et des Pays de l'Adour (UPPA), 64000 Pau, France}
\address[address-CS]{Department of Mathematics and Computer Science,\\ University of Calabria, Rende (CS), Italy}
\address[address-Como]{University of Insubria, Department of Science and High Technology,\\ via Valleggio 11, 22100 Como, Italy}
\address[address-Como2]{University of Uppsala, Department of Information Technology,\\ hus 10, L\"{a}gerhyddsv\"{a}gen 1, 75105 Uppsala (Sweden)}

\nonumnote{\ensuremath{\dagger}\ Retired from Universit\'{e} de Pau et des Pays de l'Adour on 1 September 2025; this work was carried out during retirement; currently an independent researcher.}

\begin{abstract}
In this paper, we first study the singular value distribution of the unscaled matrices arising from Jacobi weighted histopolation. For the Jacobi weight with parameters $(\alpha,\beta)$, and using the Jacobi basis with parameters $(\alpha+1,\beta+1)$, we previously proved that, for $\alpha,\beta>1/2$ and quasi-uniform meshes, the growth of these matrices in the sense of singular value distribution is slower than any positive power of the logarithm of the matrix size. However, the numerical behavior of the unscaled matrices suggested a nontrivial singular value symbol, which we identify explicitly in this paper. Starting from the standard asymptotic formula for Jacobi polynomials, we derive an asymptotic representation of the weighted cell averages defining the entries of the histopolation matrices. This representation allows us to analyze the associated Gram sequence by means of GLT theory. We identify its GLT symbol and, consequently, we obtain the singular value distribution of the histopolation matrices for any $\alpha,\beta>-1$. We also characterize the set $\mathcal Z$ where the resulting symbol vanishes and compute its measure, showing that this set depends only on the mesh, while the Jacobi parameters affect the values of the symbol on its support. When this measure is positive, the singular value distribution implies the discrete asymptotic ill-posedness of the problem, while a comparison with the Toeplitz setting suggests exponential growth of the spectral condition number. When the measure of $\mathcal Z$ is zero, the problem is well-posed and the observed conditioning is algebraic in the matrix size.  Finally, numerical experiments supporting the theoretical findings are reported and discussed, while conclusions are given in the final section of the current work together with a few open problems.  
\end{abstract}


\end{frontmatter}

\section{Introduction}\label{sec:intro}
In many applications, the available data are given by integral quantities rather than by point evaluations. This naturally leads to reconstruction procedures based on integral functionals and, in particular, to histopolation methods. In the weighted case, the data are defined through weighted averages and the corresponding reconstruction problem depends on the choice of both the weight and the polynomial basis. Several weighted histopolation methods have been recently considered in~\cite{BruniErbFekete,DellAccio:2026:AGP,DellAccio:2026:NAM,Nudo:2026:FRU,Guessab:2026:QWH}, while related polynomial histopolation problems have been studied in~\cite{Bruno:2024:PIO,Bruni:2025:PHO,Bruno:2025:OTC}. In~\cite{Guessab:2026:SJW}, we considered a weighted histopolation problem on $[-1,1]$ associated with the Jacobi weight with parameters $\alpha,\beta>-1$. Starting from weighted cell averages, we expressed the polynomial histopolant in the Jacobi basis with parameters $(\alpha+1,\beta+1)$ and derived an exact factorization of the corresponding histopolation matrix. This representation allowed us to perform the asymptotic analysis
by means of GLT theory and to show that, for $\alpha,\beta>1/2$
and quasi-uniform meshes, the growth of the histopolation matrices
in the sense of singular value distribution is slower than any positive
power of the logarithm of the matrix size. The numerical experiments for the unscaled matrices provided further information on the sequence and suggested a nontrivial limiting singular value distribution, with a sharp decay in the final part of the spectrum. This observation is the main motivation for the present paper. We therefore consider the unscaled matrix sequence and explicitly identify its singular value symbol. The analysis starts from the classical asymptotic formula for Jacobi polynomials. When the polynomial degree grows proportionally to the matrix size, the classical asymptotic formula for Jacobi polynomials yields an asymptotic representation of the weighted cell averages. This representation allows us to study the associated Gram sequence within the GLT framework and to identify its symbol explicitly.
 As a consequence, we obtain the singular value distribution of the unscaled histopolation matrices. The resulting description also explains the sharp decay observed in the final part of the spectra. More specifically, when the singular value symbol is not sparsely
vanishing, i.e., when the set of its zeros $\mathcal Z$ has positive measure, two
closely related features become relevant:
\begin{itemize}
\item[-] the singular value distribution implies the discrete asymptotic
ill-posedness of the problem;
\item[-] the spectral condition number exhibits numerically an
exponential trend as the matrix size increases.
\end{itemize}
The first property follows directly from the singular value distribution,
whereas the second is motivated by the spectral behavior associated with
the extended zero set, as in the Toeplitz setting \cite{Capizzano:2000:HBC,Tilli:2003:UBO}, and is investigated numerically in
Section~\ref{sec:numerical}. Conversely, when the measure of $\mathcal Z$ is zero, the problem is well-posed 
and the observed spectral conditioning is algebraic in the matrix size.
Hence, from a numerical linear algebra perspective, it remains to understand how to handle the histopolation problem
efficiently and accurately for large problem sizes, in the difficult case of $\mathcal Z$ having positive measure. 

The paper is organized as follows. In Section~\ref{sec:setting}, we introduce the weighted histopolation matrices and state the main result. In Section~\ref{sec:entries}, we derive the asymptotic representation of the matrix entries used in the spectral analysis. Section~\ref{sec:gram} is devoted to the associated Gram sequence, whose GLT symbol is identified explicitly and yields the singular value distribution of the unscaled histopolation matrices. In Section~\ref{sec:further}, we discuss further properties of the zero
set of the symbol and its dependence on the mesh.
Section~\ref{sec:numerical} presents numerical experiments for different
meshes that support the theoretical results.
Finally, we summarize the main results in Section~\ref{sec:conclusions}.

\section{Jacobi weighted histopolation matrices and the main result}\label{sec:setting}

Let $\alpha,\beta>-1$ and consider the Jacobi weight
\begin{equation*}
\omega_{\alpha,\beta}(x)=(1-x)^\alpha(1+x)^\beta,
\quad x\in[-1,1].
\end{equation*}
Let $g\in C^3([0,1])$ satisfy
\begin{equation}\label{eq:g-ass}
g(0)=-1,\quad g(1)=1,\quad g'(y)>0,
\quad y\in(0,1).
\end{equation}
For $N\ge1$, we define the node set
\begin{equation}\label{eq:nodes}
X_N=\left\{x_0,\dots,x_N\right\}, \quad \text{where}\quad x_i=g\left(\frac{i}{N}\right),
\quad i=0,\ldots,N,
\end{equation}
and the subintervals
\[
s_i=\left[x_{i-1},x_i\right],\quad
h_i=x_i-x_{i-1},
\quad i=1,\ldots,N.
\]
We consider the Jacobi basis
\begin{equation*}
 \varphi_{j-1}(x)=P_{j-1}^{(\alpha+1,\beta+1)}(x),
 \quad j=1,\ldots,N.
\end{equation*}
The corresponding weighted histopolation matrix $H_N\in\RR^{N\times N}$ is defined by
\begin{equation}\label{eq:HN}
 [H_N]_{i,j}
 =
 \frac{1}{h_i}
 \int_{x_{i-1}}^{x_i}
 P_{j-1}^{(\alpha+1,\beta+1)}(x)
 \omega_{\alpha,\beta}(x)\dd x,
 \quad i,j=1,\ldots,N.
\end{equation}
For any $N$, the matrix $H_N$ is nonsingular; see \cite[Thm. 2.3]{Guessab:2026:SJW}. To state the main result, we define
\begin{equation}\label{eq:vartheta-q}
 \vartheta(y)=\arccos(g(y)),\quad
 q(y)=-\vartheta'(y)=\frac{g'(y)}{\sqrt{1-(g(y))^2}},
 \quad y\in(0,1),
\end{equation}
and
\begin{equation}\label{chi:eq}
    \chi(t)=
\begin{cases}
\dfrac{\sin t}{t}, & t\neq 0,\\[1ex]
1, & t=0.
\end{cases}
\end{equation}
We then define, for a.e. $(y,\theta)\in(0,1)\times[-\pi,\pi]$,
the nonnegative measurable function
\begin{equation}\label{eq:kappa-main}
\kappa_{\alpha,\beta,g}(y,\theta)
=
2^{2\alpha+2\beta-1}
\left(\sin\left(\frac{\vartheta(y)}{2}\right)\right)^{2\alpha-3}
\left(\cos\left(\frac{\vartheta(y)}{2}\right)\right)^{2\beta-3}
\sum_{\substack{k\in\mathbb Z\\
0<|\theta+2k\pi|<q(y)}}
\frac{
\left[\chi\left(\dfrac{|\theta+2k\pi|}{2}\right)\right]^2
}{
|\theta+2k\pi|
}.
\end{equation}
We recall that, for a sequence of matrices
\[
\left\{A_N\right\}_N,\quad A_N\in\mathbb{C}^{N\times N},
\]
and a measurable function $f$ on $[0,1]\times[-\pi,\pi]$, the notation
\[
\left\{A_N\right\}_N\sim_\sigma f
\]
means that, for any $F\in C_c(\RR)$, we have
\[
\lim_{N\to\infty}\frac{1}{N}\sum_{j=1}^N
F\left(\sigma_j\left(A_N\right)\right)
=
\frac{1}{2\pi}
\int_0^1\int_{-\pi}^{\pi}
F\left(|f(y,\theta)|\right)\dd\theta\dd y.
\]
For the results proved in this paper, we only need the standard GLT notation, the basic algebraic properties of GLT sequences, and the approximating class of sequences as presented in~\cite{Tilli:1998:LTS, Serra:2003:GLT, Garoni:2017:GLT1,Garoni:2018:GLT2,Serra:2006:TGC,Barbarino:2020:BGL2}.

\begin{theorem}\label{thm:main}
Let $\alpha,\beta>-1$, and assume that $g$ satisfies~\eqref{eq:g-ass}. Then
\begin{equation}\label{eq:main-sv}
 \left\{H_N\right\}_N\sim_\sigma \sqrt{\kappa_{\alpha,\beta,g}}.
\end{equation}
Moreover, the Gram sequence satisfies
\begin{equation}\label{eq:gram-glt-main}
 \left\{H_NH_N^\top\right\}_N\sim_{\mathrm{GLT}}\kappa_{\alpha,\beta,g}.
\end{equation}
\end{theorem}

The proof of Theorem~\ref{thm:main} relies on several preliminary results, which are established in the following sections. Before turning to the proof, we briefly describe the structure of the symbol in~\eqref{eq:kappa-main}. The formula shows separately the role of the mesh and of the Jacobi parameters. The dependence on the mesh is described by $g(y)$, or equivalently by $\vartheta(y)$ and by $q(y)$. The condition
\[
0<|\theta+2k\pi|<q(y)
\]
determines the region where the symbol is nonzero, while the Jacobi parameters affect its values on this region.

\begin{corollary}\label{cor:zero-set}
Under the assumptions of Theorem~\ref{thm:main}, the set of zeros of
$\kappa_{\alpha,\beta,g}$ is
\begin{equation}\label{eq:zero-set}
\mathcal Z
=
\left\{(y,\theta)\in(0,1)\times[-\pi,\pi]\,:\,
|\theta|\ge q(y)\right\},
\end{equation}
up to a set of measure zero.
Moreover
\begin{equation}\label{eq:zero-mu}
\frac{|\mathcal Z|}{2\pi}
=\int_0^1
\left(1-\frac{q(y)}{\pi}\right)_+
\dd y.
\end{equation}
\end{corollary}

\begin{proof}
Fix $y\in(0,1)$ and $\theta\in[-\pi,\pi]\setminus\{0\}$. Since, for any
$k\in\mathbb Z$, we have
\[
\sin\left(\frac{\theta+2k\pi}{2}\right)
=
(-1)^k\sin\left(\frac{\theta}{2}\right),
\]
it follows from~\eqref{chi:eq} that
\[
\left[\chi\left(\frac{\theta+2k\pi}{2}\right)\right]^2>0,
\quad k\in\mathbb Z.
\]
Moreover, $g(y)\in(-1,1)$ for $y\in(0,1)$, and hence
\[
\vartheta(y)=\arccos(g(y))\in(0,\pi).
\]
Therefore, we have
\[
\sin\left(\frac{\vartheta(y)}{2}\right)>0,
\quad
\cos\left(\frac{\vartheta(y)}{2}\right)>0.
\]
Hence, for $y\in(0,1)$ and
$\theta\in[-\pi,\pi]\setminus\{0\}$, all the terms in
\eqref{eq:kappa-main} are strictly positive whenever the sum is
nonempty. It follows that
$\kappa_{\alpha,\beta,g}(y,\theta)>0$ if and only if
\[
\left\{
k\in\mathbb Z\,:\,
0<|\theta+2k\pi|<q(y)
\right\}\neq \emptyset.
\]
For $\theta\in[-\pi,\pi]\setminus\{0\}$, we have
\[
\min_{k\in\mathbb Z}|\theta+2k\pi|=|\theta|.
\]
Hence, the above set is nonempty if and only if $|\theta|<q(y)$.
Consequently, for almost every $(y,\theta)\in(0,1)\times[-\pi,\pi]$, we have
\[
\kappa_{\alpha,\beta,g}(y,\theta)=0
\iff
|\theta|\ge q(y).
\]
This proves~\eqref{eq:zero-set}.

It remains to determine the measure of the set~\eqref{eq:zero-set}.
For a fixed $y\in(0,1)$, the condition $|\theta|<q(y)$ defines an interval
of length
\[
2\min\{q(y),\pi\}
\]
inside $[-\pi,\pi]$. Hence,
\[
\left|
\left\{\theta\in[-\pi,\pi]\,:\,|\theta|\ge q(y)\right\}
\right|
=
2\pi-2\min\{q(y),\pi\}.
\]
After normalization by $2\pi$, we obtain
\[
1-\min\left\{1,\frac{q(y)}{\pi}\right\}
=
\left(1-\frac{q(y)}{\pi}\right)_+.
\]
Integrating with respect to $y\in(0,1)$ yields~\eqref{eq:zero-mu}.
\end{proof}

In the following, we set
\begin{equation}\label{mu0eq}
\mu_0
=
\int_0^1
\left(1-\frac{q(y)}{\pi}\right)_+
\dd y.
\end{equation}

\begin{remark}
We observe that $\mu_0$ is not necessarily positive. For instance, let
\begin{equation}\label{cosinfun}
    g(y)=-\cos(\pi y).
\end{equation}
Then
\[
q(y)
=
\frac{g'(y)}{\sqrt{1-(g(y))^2}}
=
\pi,
\quad y\in(0,1),
\]
and hence, by~\eqref{mu0eq},
\[
\mu_0=0.
\]
\end{remark}

The following proposition shows that the cosine mesh defined by~\eqref{cosinfun} is the only admissible mesh for which $\mu_0=0$.

\begin{proposition}\label{prop:cosine-mesh}
Let $g\in C^3([0,1])$ satisfy~\eqref{eq:g-ass}, and let $\mu_0$ be
defined by~\eqref{mu0eq}. Then $\mu_0=0$ if and only if
\[
g(y)=-\cos(\pi y),
\quad y\in[0,1].
\]
In particular, $\mu_0>0$ for any other admissible mesh.
\end{proposition}
\begin{proof}
For the cosine mesh, the equality $\mu_0=0$ follows from the previous
remark. Conversely, assume that
\[
\mu_0
=
\int_0^1
\left(1-\frac{q(y)}{\pi}\right)_+
\dd y
=
0.
\]
Since the integrand is nonnegative and continuous on $(0,1)$, it must
vanish at every point of this interval. Therefore,
\begin{equation}\label{prop1}
    q(y)\ge\pi,
\quad y\in(0,1).
\end{equation}
On the other hand, since 
\[
q(y)=-\vartheta'(y), \quad y\in (0,1),
\]
and $\vartheta$ is
continuous on $[0,1]$, with
\[
\vartheta(0)=\pi,
\quad
\vartheta(1)=0,
\]
we obtain
\begin{equation}\label{prop2}
    \int_0^1 q(y)\dd y
=-\int_0^1 \vartheta'(y)\dd y
=
\pi.
\end{equation}
The two properties~\eqref{prop1} and~\eqref{prop2} imply that $q$ is identically equal to $\pi$ on
$(0,1)$. Indeed, if $q(y_0)>\pi$ at some point $y_0\in(0,1)$, continuity
would give $q>\pi$ on an interval containing $y_0$. Together with
$q\ge\pi$ on $(0,1)$, this would imply
\[
\int_0^1 q(y)\dd y>\pi,
\]
which is a contradiction. Hence,
\[
q(y)=\pi,
\quad y\in(0,1).
\]
Therefore, using that $\vartheta(0)=\pi$, we obtain
\[
\vartheta(y)-\pi=\int_0^y \vartheta'(s)\dd s=-\int_{0}^y q(s)\dd s=-\pi \int_{0}^y \dd s=-\pi y.
\]
Consequently
\[
\vartheta(y)= \pi (1-y)
\]
and then
\[
g(y)
=
\cos(\vartheta(y))
=
\cos(\pi(1-y))
=
-\cos(\pi y),
\quad y\in[0,1].
\]
This proves the equivalence. Since $\mu_0\ge0$, it also follows that
$\mu_0>0$ for any other admissible mesh.
\end{proof}

\section{Asymptotic representation}\label{sec:entries}

To prove Theorem~\ref{thm:main}, we need some preliminary notation and the classical asymptotic formula for Jacobi polynomials. For fixed $a,b>-1$ and $\theta\in(0,\pi)$, we have
\begin{equation}\label{eq:jacobi-asymptotic}
P_n^{(a,b)}(\cos\theta)
=
\frac{
\left(\sin\left(\frac{\theta}{2}\right)\right)^{-a-\frac12}
\left(\cos\left(\frac{\theta}{2}\right)\right)^{-b-\frac12}
}{\sqrt{\pi n}}
\cos\left(
\left(n+\frac{a+b+1}{2}\right)\theta
-\frac{2a+1}{4}\pi
\right)
+O\left(n^{-3/2}\right),
\end{equation}
uniformly for $\theta$ in compact subsets of $(0,\pi)$; see, e.g.,
\cite{Szego:1975:OP,habermann2021semicircle}. For $i,j=1,\ldots,N$, we set
\begin{equation*}
y_i=\frac{i-\frac12}{N},
\quad
\eta_j=\frac{j}{N},
\quad i,j=1,\ldots,N.
\end{equation*}
We define
\begin{equation}\label{eq:phase}
 \phi_{\alpha,\beta}(y)
 =
 \frac{\alpha+\beta+1}{2}\vartheta(y)
 -\frac{2\alpha+3}{4}\pi,
\end{equation}
and
\begin{equation}\label{eq:Aamp}
A_{\alpha,\beta,g}(y,\eta)
=
\frac{2^{\alpha+\beta}}{\sqrt{\pi}}
\eta^{-1/2}
\left(\sin\left(\frac{\vartheta(y)}{2}\right)\right)^{\alpha-\frac{3}{2}}
\left(\cos\left(\frac{\vartheta(y)}{2}\right)\right)^{\beta-\frac{3}{2}}
\chi\left(\frac{\eta q(y)}{2}\right).
\end{equation}
The previous notation allows us to derive the asymptotic representation of
the entries of $H_N$ that will be used in the sequel.

\begin{lemma}\label{lem:entry-asymptotic}
Let $0<\varepsilon<1/2$ and assume that
\begin{equation}\label{assuap1}
    \varepsilon\le y_i\le 1-\varepsilon,
\quad
\varepsilon\le \eta_j\le 1, \quad i,j=1,\ldots,N.
\end{equation}
Then
\begin{equation}
\left[H_N\right]_{i,j}
=
\frac1{\sqrt N}
A_{\alpha,\beta,g}\left(y_i,\eta_j\right)
\cos\left(
j\vartheta\left(y_i\right)+\phi_{\alpha,\beta}\left(y_i\right)
\right)
+O\left(N^{-3/2}\right),
\end{equation}
uniformly with respect to the indices $i,j$ satisfying~\eqref{assuap1}.
\end{lemma}

\begin{proof}
Using the change of variable $x=g(y)$ in~\eqref{eq:HN}, we obtain
\begin{equation}\label{eq:HN-y}
\left[H_N\right]_{i,j}
=
\frac1{h_i}
\int_{y_i-\frac1{2N}}^{y_i+\frac1{2N}}
P_{j-1}^{(\alpha+1,\beta+1)}(g(y))
\omega_{\alpha,\beta}(g(y))g'(y)\dd y.
\end{equation}
The integration is performed over
\begin{equation}\label{inty}
    y\in\left[y_i-\frac1{2N},y_i+\frac1{2N}\right],
\end{equation}
and hence, by~\eqref{assuap1}, we have
\[
\varepsilon-\frac1{2N}
\le y\le
1-\varepsilon+\frac1{2N}.
\]
Therefore, for $N$ sufficiently large, the integration interval is contained in a compact subset of $(0,1)$. Since $g(y)\in(-1,1)$ for $y\in(0,1)$, it follows that
\[
\vartheta(y)=\arccos(g(y))
\]
ranges in a compact subset of $(0,\pi)$ throughout the integration interval. Moreover,
\[
\eta_j=\frac{j}{N}\ge\varepsilon
\]
implies
\begin{equation*}
    j\ge\varepsilon N.
\end{equation*}
Then, since the Jacobi asymptotic formula~\eqref{eq:jacobi-asymptotic} is uniform for $\vartheta$ in compact subsets of $(0,\pi)$, it holds uniformly with respect to $y$ throughout the integration interval whenever $j\ge\varepsilon N$. In particular, since
\[
1-g(y)
=
2\left(\sin\left(\frac{\vartheta(y)}2\right)\right)^2,
\quad
1+g(y)
=
2\left(\cos\left(\frac{\vartheta(y)}2\right)\right)^2,
\]
applying~\eqref{eq:jacobi-asymptotic}, for
\[
a=\alpha+1,
\quad
b=\beta+1,
\]
we have
\begin{eqnarray}
&& \notag P_{j-1}^{(\alpha+1,\beta+1)}(g(y))
\omega_{\alpha,\beta}(g(y))=\\&&
\frac{2^{\alpha+\beta}}{\sqrt{\pi j}}
\left(\sin\left(\frac{\vartheta(y)}2\right)\right)^{\alpha-\frac32}
\left(\cos\left(\frac{\vartheta(y)}2\right)\right)^{\beta-\frac32}
\cos\left(
j\vartheta(y)+\phi_{\alpha,\beta}(y)
\right)
+O\left(j^{-3/2}\right).
\label{eq:weighted-asymptotic}
\end{eqnarray}
Using~\eqref{inty}, we can write
\[
y=y_i+\frac{u}{N},
\quad
-\frac{1}{2}\le u\le\frac{1}{2}.
\]
Since $\vartheta'(y)=-q(y)$ by~\eqref{eq:vartheta-q}, the Taylor
expansion gives
\begin{equation}\label{varopas}
    \vartheta\left(y_i+\frac{u}{N}\right)
=
\vartheta\left(y_i\right)
-\frac{q\left(y_i\right)}{N}u
+O\left(N^{-2}\right),
\end{equation}
uniformly for $|u|\le1/2$. Consequently, by~\eqref{eq:phase}, we have
\begin{equation}\label{varopas1}
\phi_{\alpha,\beta}\left(y_i+\frac{u}{N}\right)-
\phi_{\alpha,\beta}\left(y_i\right)=\frac{\alpha+\beta+1}{2}\left[\vartheta\left(y_i+\frac{u}{N}\right)-\vartheta\left(y_i\right)\right]=O\left(N^{-1}\right).
\end{equation}
Then, multiplying~\eqref{varopas} by $j$ and combining it with~\eqref{varopas1}, using $j/N=\eta_j$ and $j\le N$, we obtain
\begin{equation}\label{sammmasasa}
    j\vartheta\left(y_i+\frac{u}{N}\right)
+
\phi_{\alpha,\beta}\left(y_i+\frac{u}{N}\right)
=
j\vartheta\left(y_i\right)
+
\phi_{\alpha,\beta}\left(y_i\right)
-
\eta_jq\left(y_i\right)u
+
O\left(N^{-1}\right).
\end{equation}
A Taylor expansion around $y_i$ also gives
\begin{equation}\label{eq:hi-mid}
h_i=x_i-x_{i-1}
=
g\left(y_i+\frac1{2N}\right)
-
g\left(y_i-\frac1{2N}\right)
=
\frac{g'\left(y_i\right)}{N}+O\left(N^{-3}\right).
\end{equation}
Substituting~\eqref{eq:weighted-asymptotic} into~\eqref{eq:HN-y} and using the change of variable
\[
u\in[-1/2,1/2]\to y=y_i+\frac{u}{N},
\quad
\dd y=\frac{\dd u}{N},
\]
and recalling that
\[
\varepsilon N\le j\le N,
\]
we obtain
\begin{eqnarray*}
\left[H_N\right]_{i,j}
&=&
\frac{2^{\alpha+\beta}}{\sqrt{\pi j}}
\frac{1}{Nh_i}
\int_{-1/2}^{1/2}
\left(\sin\left(\frac{\vartheta\left(y_i+\frac{u}{N}\right)}{2}\right)\right)^{\alpha-\frac32}
\left(\cos\left(\frac{\vartheta\left(y_i+\frac{u}{N}\right)}{2}\right)\right)^{\beta-\frac32}
\\
&&\times
g'\left(y_i+\frac{u}{N}\right)
\cos\left(
j\vartheta\left(y_i+\frac{u}{N}\right)
+\phi_{\alpha,\beta}\left(y_i+\frac{u}{N}\right)
\right)\dd u
+O\left(N^{-3/2}\right).
\end{eqnarray*}
We notice that, by~\eqref{eq:hi-mid}, we get
\begin{equation}\label{sacca}
    Nh_i
=
g'\left(y_i\right)+O\left(N^{-2}\right).
\end{equation}
Using the mean value theorem, uniformly for $|u|\le1/2$, we also have
\begin{equation}\label{sacca1}
g'\left(y_i+\frac{u}{N}\right)
=
g'\left(y_i\right)+O\left(N^{-1}\right).
\end{equation}
Since $g'\left(y_i\right)$ is bounded away from zero for
$\varepsilon\le y_i\le1-\varepsilon$, combining~\eqref{sacca} and~\eqref{sacca1}, it follows that
\begin{equation}\label{pass1}
    \frac{g'\left(y_i+\frac{u}{N}\right)}{Nh_i}=\frac{g'\left(y_i\right)+O\left(N^{-1}\right)}{g'\left(y_i\right)+O\left(N^{-2}\right)}
=
1+O\left(N^{-1}\right).
\end{equation}
Applying again the mean value theorem, uniformly for $|u|\le1/2$, we obtain
\begin{eqnarray}
&&\notag
\left(\sin\left(\frac{\vartheta\left(y_i+\frac{u}{N}\right)}{2}\right)\right)^{\alpha-\frac32}
\left(\cos\left(\frac{\vartheta\left(y_i+\frac{u}{N}\right)}{2}\right)\right)^{\beta-\frac32}
\\
&=&
\left(\sin\left(\frac{\vartheta\left(y_i\right)}{2}\right)\right)^{\alpha-\frac32}
\left(\cos\left(\frac{\vartheta\left(y_i\right)}{2}\right)\right)^{\beta-\frac32}
+O\left(N^{-1}\right).\label{pass2}
\end{eqnarray}
Using~\eqref{sammmasasa} and the mean value theorem, we also have
\begin{equation}\label{pass3}
    \cos\left(
j\vartheta\left(y_i+\frac{u}{N}\right)
+\phi_{\alpha,\beta}\left(y_i+\frac{u}{N}\right)
\right)
=
\cos\left(
j\vartheta\left(y_i\right)
+\phi_{\alpha,\beta}\left(y_i\right)
-\eta_jq\left(y_i\right)u
\right)
+O\left(N^{-1}\right).
\end{equation}
Therefore, combining~\eqref{pass1}, \eqref{pass2} and~\eqref{pass3}, we get
\begin{eqnarray*}
\left[H_N\right]_{i,j}
&=&
\frac{2^{\alpha+\beta}}{\sqrt{\pi j}}
\left(\sin\left(\frac{\vartheta\left(y_i\right)}{2}\right)\right)^{\alpha-\frac32}
\left(\cos\left(\frac{\vartheta\left(y_i\right)}{2}\right)\right)^{\beta-\frac32}
\\
&&\times
\int_{-1/2}^{1/2}
\cos\left(
j\vartheta\left(y_i\right)
+\phi_{\alpha,\beta}\left(y_i\right)
-\eta_jq\left(y_i\right)u
\right)\dd u
+O\left(N^{-3/2}\right).
\end{eqnarray*}
Finally, we have
\begin{eqnarray*}
&&
\int_{-1/2}^{1/2}
\cos\left(
j\vartheta\left(y_i\right)
+\phi_{\alpha,\beta}\left(y_i\right)
-\eta_jq\left(y_i\right)u
\right)\dd u
\\
&=&
\frac{
\sin\left(
j\vartheta\left(y_i\right)
+\phi_{\alpha,\beta}\left(y_i\right)
+\dfrac{\eta_jq\left(y_i\right)}{2}
\right)
-
\sin\left(
j\vartheta\left(y_i\right)
+\phi_{\alpha,\beta}\left(y_i\right)
-\dfrac{\eta_jq\left(y_i\right)}{2}
\right)
}{
\eta_jq\left(y_i\right)
}
\\
&=&
\frac{
2\sin\left(\dfrac{\eta_jq\left(y_i\right)}{2}\right)
}{
\eta_jq\left(y_i\right)
}
\cos\left(
j\vartheta\left(y_i\right)
+\phi_{\alpha,\beta}\left(y_i\right)
\right)
\\
&=&
\chi\left(\frac{\eta_jq\left(y_i\right)}{2}\right)
\cos\left(
j\vartheta\left(y_i\right)
+\phi_{\alpha,\beta}\left(y_i\right)
\right),
\end{eqnarray*}
where in the last equality we have used~\eqref{chi:eq}. Hence, since
\[
j^{-1/2}=N^{-1/2}\eta_j^{-1/2},
\]
the result follows from~\eqref{eq:Aamp}.
\end{proof}

\begin{remark}\label{rem:index-count}
Let $\varepsilon\in(0,1/2)$. We  observe that the number of points $y_i$ contained in the interval $[\varepsilon,1-\varepsilon]$ can be determined explicitly. Indeed, since
\[
y_i=\frac{i-\frac12}{N},
\]
we have
\[
y_i<\varepsilon
\iff
i<\varepsilon N+\frac12.
\]
Hence, there are
\[
\left\lceil\varepsilon N+\frac12\right\rceil-1
\]
points $y_i$ to the left of $\varepsilon$. Moreover, since
\[
y_{N+1-i}=1-y_i,
\]
the same number of points lies to the right of $1-\varepsilon$. Therefore, the number of points $y_i$ contained in $[\varepsilon,1-\varepsilon]$ is
\[
N-2\left(
\left\lceil\varepsilon N+\frac12\right\rceil-1
\right).
\]
Similarly, since
\[
\eta_j=\frac{j}{N},
\]
we have
\[
\eta_j<\varepsilon
\iff
j<\varepsilon N.
\]
Thus, there are
\[
\left\lceil\varepsilon N\right\rceil-1
\]
points $\eta_j$ to the left of $\varepsilon$, and hence the number of points $\eta_j$ contained in $[\varepsilon,1]$ is
\[
N-\left\lceil\varepsilon N\right\rceil+1.
\]
\end{remark}

\section{The Gram sequence}\label{sec:gram}

We first introduce a matrix sequence that will be used to study the asymptotic behavior of the Gram sequence
\[
\left\{H_NH_N^{\top}\right\}_N.
\]
To this aim, let
\[
I=[a,b]\subset(0,1),
\quad
\eta_0\in(0,1).
\]
We consider
\begin{equation}\label{hyp4}
    \vartheta\in C^2(I),
\quad
0<\vartheta(y)<\pi,
\quad
q(y)=-\vartheta'(y)>0,
\quad y\in I,
\end{equation}
and
\[
A\in C^1\left(I\times\left[\eta_0,1\right]\right),
\quad
\phi\in C^1(I).
\]

\begin{remark}
The functions introduced in Section~\ref{sec:entries} satisfy the above assumptions on any compact interval $I\subset(0,1)$. Indeed, since $g\in C^3([0,1])$ and $g(y)\in(-1,1)$ for $y\in I$, we have
\[
\vartheta=\arccos(g(y))\in C^2(I),
\quad
0<\vartheta(y)<\pi,
\quad
q(y)=-\vartheta'(y)>0.
\]
Moreover, since $I\subset(0,1)$, we have
\[
g(y)\in(-1,1),
\]
and hence
\[
\frac{\vartheta(y)}{2}
=
\frac{\arccos(g(y))}{2}
\in\left(0,\frac{\pi}{2}\right).
\]
Therefore,
\[
\sin\left(\frac{\vartheta(y)}{2}\right)>0,
\quad
\cos\left(\frac{\vartheta(y)}{2}\right)>0,
\quad y\in I.
\]
Since $I$ is compact, both functions are bounded away from zero on $I$. Hence, for any $\eta_0\in(0,1)$, all the factors appearing in~\eqref{eq:Aamp} are continuously differentiable on
$I\times\left[\eta_0,1\right]$, and therefore
\[
A_{\alpha,\beta,g}\in
C^1\left(I\times\left[\eta_0,1\right]\right).
\]
\end{remark}

 Let $Q_N\in\RR^{N\times N}$ be defined by
\begin{equation}\label{eq:QN-general}
\left[Q_N\right]_{i,j}
=
\begin{cases}
\dfrac{1}{\sqrt{N}}
A\left(y_i,\eta_j\right)
\cos\left(
j\vartheta\left(y_i\right)+\phi\left(y_i\right)
\right),
&
\left(y_i,\eta_j\right)\in I\times\left[\eta_0,1\right],
\\[2ex]
0,
&
\text{otherwise}.
\end{cases}
\end{equation}
For $y\in I$ and $\theta\in[-\pi,\pi]$, we define
\begin{equation}\label{eq:kappa-general}
\kappa_I(y,\theta)
=
\frac{\pi}{2q(y)}
\sum_{\substack{\ell\in\mathbb Z\\
\eta_0q(y)\le \left|\theta+2\ell\pi\right|\le q(y)}}
\left[A\left(
y,
\frac{\left|\theta+2\ell\pi\right|}{q(y)}
\right)\right]^2,
\end{equation}
and set
\[
\kappa_I(y,\theta)=0,
\quad y\notin I.
\]

\begin{proposition}\label{prop:gram-general}
The matrix sequence $\left\{Q_NQ_N^\top\right\}_N$ satisfies
\begin{equation}\label{eq:gram-general-glt}
\left\{Q_NQ_N^\top\right\}_N
\sim_{\mathrm{GLT}}
\kappa_I.
\end{equation}
\end{proposition}
\begin{proof}
We set $G_N=Q_NQ_N^\top$ and
\[
m_N=\left\lceil N\eta_0\right\rceil.
\]
For $y_i,y_k\in I$, using the trigonometric identity
\[
\cos(a)\cos(b)
=
\frac{1}{2}\cos\left(a-b\right)
+
\frac{1}{2}\cos\left(a+b\right),
\]
and recalling that
\[
\eta_j=\frac{j}{N},
\quad
\eta_j\in\left[\eta_0,1\right]
\iff
j=m_N,\ldots,N,
\]
we obtain
\begin{eqnarray}
\left[G_N\right]_{i,k}
&=&
\frac{1}{2N}
\sum_{j=m_N}^{N}
A\left(y_i,\eta_j\right)
A\left(y_k,\eta_j\right)
\cos\left(
j\left(\vartheta\left(y_i\right)-\vartheta\left(y_k\right)\right)
+
\phi\left(y_i\right)-\phi\left(y_k\right)
\right)
\notag
\\
&&+
\frac{1}{2N}
\sum_{j=m_N}^{N}
A\left(y_i,\eta_j\right)
A\left(y_k,\eta_j\right)
\cos\left(
j\left(\vartheta\left(y_i\right)+\vartheta\left(y_k\right)\right)
+
\phi\left(y_i\right)+\phi\left(y_k\right)
\right).
\label{eq:gram-products}
\end{eqnarray}
We first analyze the entries along a fixed diagonal. Let $r\in\mathbb Z$ be fixed and set
\[
k=i-r.
\]
Since
\begin{equation}\label{yi}
    y_i-y_{i-r}=\frac{r}{N},
\end{equation}
we can write
\[
\vartheta\left(y_{i-r}\right)=\vartheta\left(y_{i}-\frac{r}{N}\right).
\]
Then, a Taylor expansion gives
\begin{equation}\label{multsp}
    N\left(
\vartheta\left(y_i\right)
-
\vartheta\left(y_{i-r}\right)
\right)
=
-rq\left(y_i\right)
+
O\left(N^{-1}\right),
\end{equation}
uniformly for $y_i,y_{i-r}\in I$. Moreover, using~\eqref{yi} and the mean value theorem, we have
\begin{equation}
\label{multsp1}
\phi\left(y_i\right)-\phi\left(y_{i-r}\right)
=
O\left(N^{-1}\right),
\end{equation}
and
\begin{equation}\label{Aimr}
    A\left(y_{i-r},\eta_j\right)
=
A\left(y_i,\eta_j\right)
+
O\left(N^{-1}\right),
\end{equation}
uniformly with respect to $j$. Then, multiplying both sides of~\eqref{multsp} by
\[
\eta_j=\frac{j}{N},
\]
and combining the resulting relation with~\eqref{multsp1}, we obtain
\[
j\left(
\vartheta\left(y_i\right)
-
\vartheta\left(y_{i-r}\right)
\right)
+
\phi\left(y_i\right)
-
\phi\left(y_{i-r}\right)
=
-rq\left(y_i\right)\eta_j
+
O\left(N^{-1}\right).
\]
Therefore, for $k=i-r$, by the last relation and~\eqref{Aimr}, the first sum in~\eqref{eq:gram-products} is equal to
\[
\frac{1}{2N}
\sum_{j=m_N}^{N}
\left[A\left(y_i,\eta_j\right)\right]^2
\cos\left(
rq\left(y_i\right)\eta_j
\right)
+
O\left(N^{-1}\right).
\]
For a fixed $r\in\mathbb Z$, the function
\[
(y,\eta)
\longmapsto
\left[A\left(y,\eta\right)\right]^2
\cos\left(rq(y)\eta\right)
\]
is continuously differentiable with respect to $\eta$ on
$I\times\left[\eta_0,1\right]$. Hence, denoting by
\begin{equation*}
c_r(y)
=
\frac{1}{2}
\int_{\eta_0}^{1}
\left[A\left(y,\eta\right)\right]^2
\cos\left(rq(y)\eta\right)
\dd\eta,
\quad y\in I,
\end{equation*}
the corresponding Riemann sums satisfy
\begin{equation}\label{amas}
    \frac{1}{2N}
\sum_{j=m_N}^{N}
\left[A\left(y_i,\eta_j\right)\right]^2
\cos\left(rq\left(y_i\right)\eta_j\right)
=
c_r\left(y_i\right)
+
O\left(N^{-1}\right),
\end{equation}
uniformly for $y_i\in I$. We set
\[
c_r(y)=0,
\quad y\notin I.
\]

We now consider the second sum in~\eqref{eq:gram-products}. Since $I$ is compact and
\begin{equation}\label{thera}
    0<\vartheta(y)<\pi,
\quad y\in I,
\end{equation}
there exists a constant $c>0$ such that
\[
c\le \vartheta(y)\le \pi-c,
\quad y\in I.
\]
Hence,
\[
2c
\le
\vartheta\left(y_i\right)+\vartheta\left(y_k\right)
\le
2\pi-2c.
\]
Therefore, we obtain
\begin{equation}\label{sinbound}
\left|
1-
e^{\mathrm{i}\left(\vartheta\left(y_i\right)+\vartheta\left(y_k\right)\right)}
\right|
=
2\left|
\sin\left(
\frac{\vartheta\left(y_i\right)+\vartheta\left(y_k\right)}{2}
\right)
\right|
\ge
2\sin\left(c\right)>0.
\end{equation}
Since $A\in C^1\left(I\times\left[\eta_0,1\right]\right)$, there exists a constant $C>0$, independent of $N,i,k$, such that
\begin{equation}\label{A1}
    \max_{m_N\le j\le N}
\left|
A\left(y_i,\eta_j\right)
A\left(y_k,\eta_j\right)
\right|
\le C
\end{equation}
and
\begin{equation}\label{A2}
\sum_{j=m_N}^{N-1}
\left|
A\left(y_i,\eta_{j+1}\right)
A\left(y_k,\eta_{j+1}\right)
-
A\left(y_i,\eta_j\right)
A\left(y_k,\eta_j\right)
\right|
\le C.
\end{equation}
Moreover, using the geometric sum, we have
\begin{eqnarray*}
\sum_{j=m_N}^{\ell}
e^{\mathrm{i}j\left(\vartheta\left(y_i\right)+\vartheta\left(y_k\right)\right)}
&=&
e^{\mathrm{i}m_N\left(\vartheta\left(y_i\right)+\vartheta\left(y_k\right)\right)}
\sum_{j=0}^{\ell-m_N}
e^{\mathrm{i}j\left(\vartheta\left(y_i\right)+\vartheta\left(y_k\right)\right)}
\\
&=&
e^{\mathrm{i}m_N\left(\vartheta\left(y_i\right)+\vartheta\left(y_k\right)\right)}
\frac{
1-
e^{\mathrm{i}\left(\ell-m_N+1\right)
\left(\vartheta\left(y_i\right)+\vartheta\left(y_k\right)\right)}
}{
1-
e^{\mathrm{i}\left(\vartheta\left(y_i\right)+\vartheta\left(y_k\right)\right)}
}.
\end{eqnarray*}
Hence,  by~\eqref{sinbound}, we obtain
\begin{equation}\label{49new}
    \left|
\sum_{j=m_N}^{\ell}
e^{\mathrm{i}j\left(\vartheta\left(y_i\right)+\vartheta\left(y_k\right)\right)}
\right|
\le
\frac{2}{
\left|
1-
e^{\mathrm{i}\left(\vartheta\left(y_i\right)+\vartheta\left(y_k\right)\right)}
\right|
}
\le
\frac{1}{\sin\left(c\right)},
\quad
m_N\le\ell\le N.
\end{equation}
We now bound the second sum in~\eqref{eq:gram-products}. Using
\[
\cos(t)=\operatorname{Re}\left(e^{\mathrm{i}t}\right),
\]
we have
\begin{eqnarray}
\notag &&
\frac{1}{2N}
\sum_{j=m_N}^{N}
A\left(y_i,\eta_j\right)
A\left(y_k,\eta_j\right)
\cos\left(
j\left(\vartheta\left(y_i\right)+\vartheta\left(y_k\right)\right)
+
\phi\left(y_i\right)+\phi\left(y_k\right)
\right)
\\
&=&
\frac{1}{2N}
\operatorname{Re}\left[
e^{\mathrm{i}\left(\phi\left(y_i\right)+\phi\left(y_k\right)\right)}
\sum_{j=m_N}^{N}
A\left(y_i,\eta_j\right)
A\left(y_k,\eta_j\right)
e^{\mathrm{i}j\left(\vartheta\left(y_i\right)+\vartheta\left(y_k\right)\right)}
\right].\label{realpart}
\end{eqnarray}
Moreover, for any $j=m_N,\ldots,N$, we can write
\begin{equation}\label{telescops}
e^{\mathrm{i}j\left(\vartheta\left(y_i\right)+\vartheta\left(y_k\right)\right)}
=
\sum_{\ell=m_N}^{j}
e^{\mathrm{i}\ell\left(\vartheta\left(y_i\right)+\vartheta\left(y_k\right)\right)}
-
\sum_{\ell=m_N}^{j-1}
e^{\mathrm{i}\ell\left(\vartheta\left(y_i\right)+\vartheta\left(y_k\right)\right)}.
\end{equation}
Then, we obtain
\begin{eqnarray*}
&&
\sum_{j=m_N}^{N}
A\left(y_i,\eta_j\right)
A\left(y_k,\eta_j\right)
e^{\mathrm{i}j\left(\vartheta\left(y_i\right)+\vartheta\left(y_k\right)\right)}
\\
&=&
\sum_{j=m_N}^{N}
A\left(y_i,\eta_j\right)
A\left(y_k,\eta_j\right)
\sum_{\ell=m_N}^{j}
e^{\mathrm{i}\ell\left(\vartheta\left(y_i\right)+\vartheta\left(y_k\right)\right)}
-
\sum_{j=m_N+1}^{N}
A\left(y_i,\eta_j\right)
A\left(y_k,\eta_j\right)
\sum_{\ell=m_N}^{j-1}
e^{\mathrm{i}\ell\left(\vartheta\left(y_i\right)+\vartheta\left(y_k\right)\right)}\\
&=&
\sum_{j=m_N}^{N}
A\left(y_i,\eta_j\right)
A\left(y_k,\eta_j\right)
\sum_{\ell=m_N}^{j}
e^{\mathrm{i}\ell\left(\vartheta\left(y_i\right)+\vartheta\left(y_k\right)\right)}
-
\sum_{j=m_N}^{N-1}
A\left(y_i,\eta_{j+1}\right)
A\left(y_k,\eta_{j+1}\right)
\sum_{\ell=m_N}^{j}
e^{\mathrm{i}\ell\left(\vartheta\left(y_i\right)+\vartheta\left(y_k\right)\right)}.
\end{eqnarray*}
Here, in the second sum of the first equality, the term corresponding to $j=m_N$ vanishes since the inner sum is empty. Therefore, by~\eqref{A1}, \eqref{A2}, \eqref{49new}, and the triangle inequality, we have
\begin{eqnarray}
\notag &&
\left|
\sum_{j=m_N}^{N}
A\left(y_i,\eta_j\right)
A\left(y_k,\eta_j\right)
e^{\mathrm{i}j\left(\vartheta\left(y_i\right)+\vartheta\left(y_k\right)\right)}
\right|
\\\notag
&\le&
\left|
A\left(y_i,\eta_N\right)
A\left(y_k,\eta_N\right)
\right|
\left|
\sum_{\ell=m_N}^{N}
e^{\mathrm{i}\ell\left(\vartheta\left(y_i\right)+\vartheta\left(y_k\right)\right)}
\right|
\\
&&+
\sum_{j=m_N}^{N-1}
\left|
A\left(y_i,\eta_j\right)
A\left(y_k,\eta_j\right)
-
A\left(y_i,\eta_{j+1}\right)
A\left(y_k,\eta_{j+1}\right)
\right|
\left|
\sum_{\ell=m_N}^{j}
e^{\mathrm{i}\ell\left(\vartheta\left(y_i\right)+\vartheta\left(y_k\right)\right)}
\right|
\le
C,\label{amslappppa}
\end{eqnarray}
uniformly with respect to $N,i,k$. Then, combining the estimates for the two sums in~\eqref{eq:gram-products} given by~\eqref{amas}, \eqref{realpart}, and~\eqref{amslappppa}, we obtain that, for any fixed $r\in\mathbb Z$ and $k=i-r$,
\[
\left[G_N\right]_{i,i-r}
=
c_r\left(y_i\right)+O\left(N^{-1}\right),
\]
uniformly for $y_i,y_{i-r}\in I$.

We next derive a uniform estimate for the entries of $G_N$. Since, by assumption, $q(y)>0$ on the compact interval $I$, there exists a constant $c>0$ such that
\[
q(y)\ge c,
\quad y\in I.
\]
Then, by~\eqref{hyp4} and the mean value theorem, we have
\[
\left|
\vartheta\left(y_i\right)
-
\vartheta\left(y_k\right)
\right|
\ge
c\frac{\left|i-k\right|}{N},
\quad i\neq k.
\]
Moreover, by~\eqref{thera}, we have
\[
-\pi
<
\vartheta\left(y_i\right)
-
\vartheta\left(y_k\right)
<
\pi.
\]
Hence, using
\[
\sin(x)\ge\frac{2}{\pi}x,
\quad \text{for any}\quad
 x\in\left[0,\frac{\pi}{2}\right],
\]
we obtain
\begin{equation}\label{smaaac}
\left|
1-
e^{\mathrm{i}\left(\vartheta\left(y_i\right)-\vartheta\left(y_k\right)\right)}
\right|
=
2\left|
\sin\left(
\frac{\vartheta\left(y_i\right)-\vartheta\left(y_k\right)}{2}
\right)
\right|
\ge
\frac{2}{\pi}
\left|
\vartheta\left(y_i\right)-\vartheta\left(y_k\right)
\right|
\ge
c\frac{\left|i-k\right|}{N}.
\end{equation}
For $i\ne k$, using~\eqref{smaaac} and proceeding as in~\eqref{49new}, we have
\begin{equation}\label{smaaacsa}
    \left|
\sum_{\ell=m_N}^{j}
e^{\mathrm{i}\ell\left(\vartheta\left(y_i\right)-\vartheta\left(y_k\right)\right)}
\right|
\le
\frac{2}{
\left|
1-
e^{\mathrm{i}\left(\vartheta\left(y_i\right)-\vartheta\left(y_k\right)\right)}
\right|
}
\le
C\frac{N}{|i-k|}, \quad m_N\le j\le N.
\end{equation}
Proceeding as in the previous case and using the same identity as in~\eqref{telescops}, we obtain
\begin{eqnarray*}
&&
\sum_{j=m_N}^{N}
A\left(y_i,\eta_j\right)
A\left(y_k,\eta_j\right)
e^{\mathrm{i}j\left(\vartheta\left(y_i\right)-\vartheta\left(y_k\right)\right)}
\\
&=&
A\left(y_i,\eta_N\right)
A\left(y_k,\eta_N\right)
\sum_{\ell=m_N}^{N}
e^{\mathrm{i}\ell\left(\vartheta\left(y_i\right)-\vartheta\left(y_k\right)\right)}
\\
&&+
\sum_{j=m_N}^{N-1}
\left[
A\left(y_i,\eta_j\right)
A\left(y_k,\eta_j\right)
-
A\left(y_i,\eta_{j+1}\right)
A\left(y_k,\eta_{j+1}\right)
\right]
\sum_{\ell=m_N}^{j}
e^{\mathrm{i}\ell\left(\vartheta\left(y_i\right)-\vartheta\left(y_k\right)\right)}.
\end{eqnarray*}
Therefore, by~\eqref{A1}, \eqref{A2} and~\eqref{smaaacsa}, we obtain
\begin{eqnarray*}
&&
\left|
\sum_{j=m_N}^{N}
A\left(y_i,\eta_j\right)
A\left(y_k,\eta_j\right)
e^{\mathrm{i}j\left(\vartheta\left(y_i\right)-\vartheta\left(y_k\right)\right)}
\right|
\\
&\le&
C\frac{N}{|i-k|}
+
C\frac{N}{|i-k|}
\sum_{j=m_N}^{N-1}
\left|
A\left(y_i,\eta_j\right)
A\left(y_k,\eta_j\right)
-
A\left(y_i,\eta_{j+1}\right)
A\left(y_k,\eta_{j+1}\right)
\right|
\le
C\frac{N}{|i-k|}.
\end{eqnarray*}
Hence, passing to the real part as in~\eqref{realpart}, the first sum in~\eqref{eq:gram-products} is bounded by
\[
\frac{C}{|i-k|}.
\]
Combining the previous estimate with~\eqref{realpart} and~\eqref{amslappppa}, we obtain
\[
\left|[G_N]_{i,k}\right|
\le
\frac{C}{|i-k|}
+
\frac{C}{N},
\quad i\ne k.
\]
Since $|i-k|\ge1$, we have
\[
\frac{1}{|i-k|}
\le
\frac{2}{1+|i-k|}.
\]
Therefore, after changing the constant $C$ if necessary, we obtain
\begin{equation}\label{eq:offdiag-bound}
\left|[G_N]_{i,k}\right|
\le C\left(
\frac{1}{1+|i-k|}
+
\frac{1}{N}\right),
\quad i\ne k,
\end{equation}
where $C$ is independent of $N,i,k$.

For $R\ge1$, we define
\[
\left[G_N^{(R)}\right]_{i,k}
=
\begin{cases}
\left[G_N\right]_{i,k}, & |i-k|\le R,\\
0, & |i-k|>R.
\end{cases}
\]
For each fixed $r\ge0$, we have shown that
\begin{equation}\label{diagGN}
    \left[G_N\right]_{i,i-r}
=
c_r\left(y_i\right)+O\left(N^{-1}\right),
\end{equation}
uniformly for $y_i,y_{i-r}\in I$. Since $G_N=Q_NQ_N^\top$ is symmetric, we also have
\[
\left[G_N\right]_{i,i+r}
=
\left[G_N\right]_{i+r,i}
=
c_r\left(y_{i+r}\right)+O\left(N^{-1}\right),
\]
uniformly for $y_i,y_{i+r}\in I$. 
For $r\ge1$, let $Z_N^r$ be the matrix defined by
\[
\left[Z_N^r\right]_{i,k}
=
\begin{cases}
1, & i=k+r,\\
0, & \text{otherwise}.
\end{cases}
\]
Moreover, let
\[
D_N\left(c_r\right)
=
\operatorname{diag}
\left[
c_r\left(y_1\right),
\ldots,
c_r\left(y_N\right)
\right].
\]
We then consider
\[
L_N^{(R)}
=
D_N\left(c_0\right)
+
\sum_{r=1}^{R}
\left[
D_N\left(c_r\right)Z_N^r
+
\left(Z_N^r\right)^\top D_N\left(c_r\right)
\right].
\]
We notice that, for $r\ge1$,
\[
\left[D_N\left(c_r\right)Z_N^r\right]_{i,i-r}
=
c_r\left(y_i\right),
\]
while
\[
\left[\left(Z_N^r\right)^\top D_N\left(c_r\right)\right]_{i,i+r}
=
c_r\left(y_{i+r}\right).
\]
Hence, like $G_N^{(R)}$, the matrix $L_N^{(R)}$ has zero entries
whenever $|i-k|>R$. More precisely,
\[
\left[L_N^{(R)}\right]_{i,k}
=
\begin{cases}
c_0\left(y_i\right), & k=i,\\[1mm]
c_r\left(y_i\right), & k=i-r,\quad 1\le r\le R,\\[1mm]
c_r\left(y_k\right), & k=i+r,\quad 1\le r\le R,\\[1mm]
0, & |i-k|>R.
\end{cases}
\]
We now estimate the difference between $G_N^{(R)}$ and $L_N^{(R)}$. If
$|i-k|>R$, then
\[
\left[G_N^{(R)}-L_N^{(R)}\right]_{i,k}=0.
\]
If $|i-k|\le R$ and $y_i,y_k\in I$, then, by~\eqref{diagGN} and the
symmetry of $G_N$, we have
\[
\left[G_N^{(R)}-L_N^{(R)}\right]_{i,k}
=
O\left(N^{-1}\right).
\]
If $|i-k|\le R$ and $y_i,y_k\notin I$, then both entries are zero. Finally, if $|i-k|\le R$ and exactly one of $y_i$ and $y_k$
belongs to $I$, then the corresponding entry of $G_N^{(R)}$
is zero, and the corresponding entry of $L_N^{(R)}$ is
uniformly bounded. Since $R$ is fixed, the number
of such entries is bounded independently of $N$. Therefore,
\begin{equation}\label{nsappaaaa}
    \frac{1}{N}
\left\|
G_N^{(R)}-L_N^{(R)}
\right\|_F^2
\le
\frac{C}{N}
\left(
\frac{1}{N}+1
\right)
\longrightarrow0,
\quad N\to\infty.
\end{equation}
Furthermore, denoting by $r=|i-k|$, from~\eqref{eq:offdiag-bound} we have
\[
\left|[G_N]_{i,k}\right|
\le
C\left(
\frac{1}{1+r}
+
\frac{1}{N}
\right).
\]
For each $r=R+1,\ldots,N-1$, there are $2(N-r)$ entries such that
$|i-k|=r$. Therefore, using~\eqref{eq:offdiag-bound} and
\[
(a+b)^2\le 2\left(a^2+b^2\right),
\]
we obtain
\begin{eqnarray*}
&&\frac{1}{N}
\left\|
G_N-G_N^{(R)}
\right\|_F^2
=
\frac{1}{N}
\sum_{|i-k|>R}
\left|[G_N]_{i,k}\right|^2
\\
&\le&
\frac{C}{N}
\sum_{r=R+1}^{N-1}
(N-r)
\left(
\frac{1}{1+r}
+
\frac{1}{N}
\right)^2
\le
\frac{C}{N}
\sum_{r=R+1}^{N-1}
(N-r)
\left(
\frac{1}{(1+r)^2}
+
\frac{1}{N^2}
\right)
\\
&=&
C
\sum_{r=R+1}^{N-1}
\frac{N-r}{N(1+r)^2}
+
\frac{C}{N^3}
\sum_{r=R+1}^{N-1}
(N-r)
\le
C
\sum_{r=R+1}^{N-1}
\frac{1}{(1+r)^2}
+
\frac{C}{N^3}
\sum_{r=1}^{N-1}
(N-r)
\\
&=&
C
\sum_{r=R+1}^{N-1}
\frac{1}{(1+r)^2}
+
\frac{C}{N^3}\frac{N(N-1)}{2}
\le
C
\sum_{r=R+1}^{N-1}
\frac{1}{r^2}
+
\frac{C}{N}.
\end{eqnarray*}
Hence,
\begin{equation}\label{eq:tail-F}
\limsup_{N\to\infty}
\frac{1}{N}
\left\|
G_N-G_N^{(R)}
\right\|_F^2
\le
C
\sum_{r=R+1}^{\infty}
\frac{1}{r^2}.
\end{equation}
Since
\[
G_N-L_N^{(R)}
=
\left(G_N-G_N^{(R)}\right)
+
\left(G_N^{(R)}-L_N^{(R)}\right),
\]
using~\eqref{nsappaaaa} and~\eqref{eq:tail-F}, we obtain
\begin{equation}\label{saoqpwnesw}
    \limsup_{N\to\infty}
\frac{1}{N}
\left\|
G_N-L_N^{(R)}
\right\|_F^2
\le
C
\sum_{r=R+1}^{\infty}
\frac{1}{r^2}.
\end{equation}
Since
\[
\lim_{R\to\infty}\sum_{r=R+1}^{\infty}
\frac{1}{r^2}=0,
\]
it follows that
\[
\lim_{R\to\infty}
\limsup_{N\to\infty}
\frac{1}{N}
\left\|
G_N-L_N^{(R)}
\right\|_F^2
=
0.
\]
Therefore, $\left\{L_N^{(R)}\right\}_N$ is an approximating class for
$\left\{G_N\right\}_N$ as $R\to\infty$; see~\cite{Barbarino:2022:RGLT}.
For any fixed $R$, we have
\[
\left\{D_N\left(c_r\right)\right\}_N
\sim_{\mathrm{GLT}}
c_r(y),
\quad
\left\{Z_N^r\right\}_N
\sim_{\mathrm{GLT}}
e^{\mathrm{i}r\theta}.
\]
Hence, by the GLT algebra, we have
\[
\left\{L_N^{(R)}\right\}_N
\sim_{\mathrm{GLT}}
\kappa_R(y,\theta),
\]
where
\[
\kappa_R(y,\theta)
=
c_0(y)
+
\sum_{r=1}^{R}
c_r(y)
\left(
e^{\mathrm{i}r\theta}
+
e^{-\mathrm{i}r\theta}
\right)
=
c_0(y)
+
2\sum_{r=1}^{R}
c_r(y)\cos(r\theta).
\]
We now identify the limit of $\kappa_R$. Fix $y\in I$ and $r\in\mathbb{Z}$. From~\eqref{eq:kappa-general}, we obtain
\begin{eqnarray*}
\frac{1}{2\pi}
\int_{-\pi}^{\pi}
\kappa_I(y,\theta)e^{-\mathrm{i}r\theta}
\dd\theta
&=&
\frac{1}{4q(y)}
\int_{\eta_0q(y)\le |t|\le q(y)}
\left[A\left(
y,\frac{|t|}{q(y)}
\right)\right]^2
e^{-\mathrm{i}rt}
\dd t
\\
&=&
\frac{1}{4}
\int_{\eta_0\le |s|\le1}
\left[A\left(y,|s|\right)\right]^2
e^{-\mathrm{i}rq(y)s}
\dd s
\\
&=&
\frac{1}{2}
\int_{\eta_0}^{1}
\left[A\left(y,\eta\right)\right]^2
\cos\left(rq(y)\eta\right)
\dd\eta
\\
&=&
c_r(y).
\end{eqnarray*}
For $y\notin I$, we have
\[
\kappa_I(y,\theta)=0,
\quad
c_r(y)=0.
\]
Therefore, for any $y\in[0,1]$, $c_r(y)$ is the $r$-th Fourier coefficient of
$\kappa_I(y,\cdot)$. Hence, $\kappa_R(y,\theta)$ is the $R$-th Fourier partial sum of $\kappa_I(y,\theta)$ with respect to
$\theta$. Since $q$ is continuous and positive on the compact interval $I$, there exist
constants $q_0,q_1>0$ such that
\[
q_0\le q(y)\le q_1,
\quad y\in I.
\]
Moreover, $A$ is bounded on $I\times\left[\eta_0,1\right]$. From~\eqref{eq:kappa-general},
the number of terms in the sum defining $\kappa_I(y,\theta)$ is bounded
independently of $y$ and $\theta$. Hence $\kappa_I$ is bounded on
$[0,1]\times[-\pi,\pi]$, and therefore
\[
\kappa_I
\in
L^2\left([0,1]\times[-\pi,\pi]\right).
\]
Since $\kappa_R$ is the $R$-th Fourier partial sum of $\kappa_I$ with respect to $\theta$, we infer
\[
\kappa_R
\longrightarrow
\kappa_I
\quad\hbox{in }
L^2\left([0,1]\times[-\pi,\pi]\right),
\quad R\to\infty,
\]
and hence also in measure. Moreover,
\[
\left\{L_N^{(R)}\right\}_N
\sim_{\mathrm{GLT}}
\kappa_R,
\]
and, using~\eqref{saoqpwnesw}, we have that $\left\{L_N^{(R)}\right\}_N$ is an approximating class of sequences
(a.c.s.) for $\left\{G_N\right\}_N$ as $R\to\infty$. Therefore, by the
a.c.s. closure property of the GLT class, we have
\[
\left\{G_N\right\}_N
\sim_{\mathrm{GLT}}
\kappa_I.
\]
This proves~\eqref{eq:gram-general-glt}.
\end{proof}

We now apply Proposition~\ref{prop:gram-general} to determine the GLT symbol of the Gram sequence associated with the Jacobi weighted histopolation matrices $\left\{H_N\right\}_N$ defined in~\eqref{eq:HN}. In view of Lemma~\ref{lem:entry-asymptotic}, we first restrict the indices to
\[
\varepsilon\le y_i\le1-\varepsilon,
\qquad
\eta_j\ge\varepsilon,
\]
with $\varepsilon\in(0,1/2)$. Let $D_{N,\varepsilon}$ be the diagonal matrix whose $i$-th diagonal entry is one if
\[
\varepsilon\le y_i\le1-\varepsilon,
\]
and zero otherwise. Analogously, let $C_{N,\varepsilon}$ be the diagonal matrix whose $j$-th diagonal entry is one if
\[
\eta_j\ge\varepsilon,
\]
and zero otherwise. We set
\begin{equation*}
H_N^{(\varepsilon)}
=
D_{N,\varepsilon}H_NC_{N,\varepsilon}.
\end{equation*}
Equivalently,
\[
\left[H_N^{(\varepsilon)}\right]_{i,j}
=
\begin{cases}
\left[H_N\right]_{i,j},
&
\varepsilon\le y_i\le1-\varepsilon
\quad\text{and}\quad
\eta_j\ge\varepsilon,
\\[1ex]
0,
&
\text{otherwise}.
\end{cases}
\]
Then, by adding and subtracting $D_{N,\varepsilon}H_N$, we have
\begin{eqnarray}
H_N-H_N^{(\varepsilon)}
&=&
H_N-D_{N,\varepsilon}H_NC_{N,\varepsilon}\notag\\
&=&
H_N-D_{N,\varepsilon}H_N
+
D_{N,\varepsilon}H_N
-
D_{N,\varepsilon}H_NC_{N,\varepsilon}
\notag\\
&=&
\left(I-D_{N,\varepsilon}\right)H_N
+
D_{N,\varepsilon}H_N\left(I-C_{N,\varepsilon}\right).\label{decs}
\end{eqnarray}
The first term contains the rows of $H_N$ corresponding to
\[
y_i\notin[\varepsilon,1-\varepsilon],
\]
while the second term contains, among the remaining rows, the columns corresponding to
\[
\eta_j<\varepsilon.
\]
By Remark~\ref{rem:index-count}, the matrix $I-D_{N,\varepsilon}$ has
\[
2\left(
\left\lceil\varepsilon N+\frac12\right\rceil-1
\right)
\]
nonzero diagonal entries, so that
\[
\operatorname{rank}\left(I-D_{N,\varepsilon}\right)
=
2\left(
\left\lceil\varepsilon N+\frac12\right\rceil-1
\right)\le
2\left\lceil\varepsilon N\right\rceil.
\]
Similarly, we deduce that
\[
\operatorname{rank}\left(I-C_{N,\varepsilon}\right)
=
\left\lceil\varepsilon N\right\rceil-1\le
\left\lceil\varepsilon N\right\rceil.
\]
Then, by~\eqref{decs}, it follows that
\begin{equation}\label{eq:rank-cut}
\operatorname{rank}\left(H_N-H_N^{(\varepsilon)}\right)
\le
2\left\lceil\varepsilon N\right\rceil+\left\lceil\varepsilon N\right\rceil=3\left\lceil\varepsilon N\right\rceil.
\end{equation}
We now consider the matrix $Q_N^{(\varepsilon)}\in\RR^{N\times N}$ defined in~\eqref{eq:QN-general} corresponding to
\[
I=[\varepsilon,1-\varepsilon],
\quad
\eta_0=\varepsilon,
\quad
A=A_{\alpha,\beta,g},
\quad
\phi=\phi_{\alpha,\beta}.
\]
By Lemma~\ref{lem:entry-asymptotic}, we have
\[
\left[H_N^{(\varepsilon)}-Q_N^{(\varepsilon)}\right]_{i,j}
=
O\left(N^{-3/2}\right),
\]
uniformly with respect to $i$ and $j$, while both entries are zero outside this region. Therefore, we have
\begin{equation*}
\left\|H_N^{(\varepsilon)}-Q_N^{(\varepsilon)}\right\|_F^2
=
\sum_{i=1}^N\sum_{j=1}^N
\left|
\left[H_N^{(\varepsilon)}-Q_N^{(\varepsilon)}\right]_{i,j}
\right|^2
=
O\left(N^{-1}\right).
\end{equation*}
Hence,
\begin{equation}\label{eq:F-small}
\left\|H_N^{(\varepsilon)}-Q_N^{(\varepsilon)}\right\|_F
=
O\left(N^{-1/2}\right).
\end{equation}
By Proposition~\ref{prop:gram-general}, we have
\begin{equation}\label{eq:Qeps-glt}
\left\{
Q_N^{(\varepsilon)}
\left(Q_N^{(\varepsilon)}\right)^\top
\right\}_N
\sim_{\mathrm{GLT}}
\kappa_\varepsilon,
\end{equation}
where, for $y\in[\varepsilon,1-\varepsilon]$,
\begin{equation}\label{eq:kappa-eps-general}
\kappa_\varepsilon(y,\theta)
=
\frac{\pi}{2q(y)}
\sum_{\substack{k\in\mathbb Z\\
\varepsilon q(y)\le
\left|\theta+2k\pi\right|
\le q(y)}}
\left[A_{\alpha,\beta,g}
\left(
y,
\frac{\left|\theta+2k\pi\right|}{q(y)}
\right)\right]^2,
\end{equation}
and
\[
\kappa_\varepsilon(y,\theta)=0,
\quad
y\notin[\varepsilon,1-\varepsilon].
\]
Substituting~\eqref{eq:Aamp} into~\eqref{eq:kappa-eps-general}, we obtain
\begin{equation*}
\kappa_\varepsilon(y,\theta)
=
2^{2\alpha+2\beta-1}
\left(\sin\left(\frac{\vartheta(y)}{2}\right)\right)^{2\alpha-3}
\left(\cos\left(\frac{\vartheta(y)}{2}\right)\right)^{2\beta-3}
\sum_{\substack{k\in\mathbb Z\\
\varepsilon q(y)\le
\left|\theta+2k\pi\right|
\le q(y)}}
\frac{
\left[\chi\left(\dfrac{|\theta+2k\pi|}{2}\right)\right]^2
}{
\left|\theta+2k\pi\right|
}.
\end{equation*}
We now let $\varepsilon\to0$. For $y\in(0,1)$ and $\varepsilon$ sufficiently small, we have
\[
y\in[\varepsilon,1-\varepsilon].
\]
Moreover, the condition
\[
\varepsilon q(y)\le|\theta+2k\pi|\le q(y)
\]
reduces, as $\varepsilon\to0$, to
\[
0<|\theta+2k\pi|\le q(y).
\]
For fixed $y\in(0,1)$, the equation
\[
|\theta+2k\pi|=q(y),
\qquad k\in\mathbb Z,
\]
can hold only for finitely many values of $\theta\in[-\pi,\pi]$, and therefore only on a set of measure zero. Therefore,
\begin{equation}\label{eq:kappa-eps-limit}
\kappa_\varepsilon(y,\theta)
\longrightarrow
\kappa_{\alpha,\beta,g}(y,\theta)
\end{equation}
for almost all $(y,\theta)\in(0,1)\times[-\pi,\pi]$ as $\varepsilon\to0$.

We are now ready to prove Theorem~\ref{thm:main}.

\begin{proof}[Proof of Theorem~\ref{thm:main}]
We first prove the statement for the sequence $\left\{H_NH_N^\top\right\}_N$. For fixed $\varepsilon\in(0,1/2)$, by~\eqref{eq:F-small}, we can write
\begin{equation}\label{HNa}
    H_N^{(\varepsilon)}
=
Q_N^{(\varepsilon)}
+
E_N^{(\varepsilon)},
\end{equation}
where
\begin{equation}\label{estimates1}
\left\|E_N^{(\varepsilon)}\right\|_F
=
O\left(N^{-1/2}\right).
\end{equation}
Moreover, the function $A_{\alpha,\beta,g}$ defined in~\eqref{eq:Aamp} is continuous on
\[
[\varepsilon,1-\varepsilon]\times[\varepsilon,1],
\]
and therefore it is bounded on this set. Hence, by~\eqref{eq:QN-general}, the entries of $Q_N^{(\varepsilon)}$ are
$O(N^{-1/2})$, uniformly with respect to $i$ and $j$. Therefore,
\[
\left\|Q_N^{(\varepsilon)}\right\|_F^2
=
\sum_{i=1}^N\sum_{j=1}^N
\left|
\left[Q_N^{(\varepsilon)}\right]_{i,j}
\right|^2
=
O(N),
\]
and hence
\begin{equation}\label{estimates2}
\left\|Q_N^{(\varepsilon)}\right\|_F
=
O\left(\sqrt N\right).
\end{equation}
Using~\eqref{HNa}, we have
\begin{eqnarray*}
H_N^{(\varepsilon)}
\left(H_N^{(\varepsilon)}\right)^\top
-
Q_N^{(\varepsilon)}
\left(Q_N^{(\varepsilon)}\right)^\top&=&\left(Q_N^{(\varepsilon)}
+
E_N^{(\varepsilon)}\right)\left(Q_N^{(\varepsilon)}
+
E_N^{(\varepsilon)}\right)^{\top}-Q_N^{(\varepsilon)}
\left(Q_N^{(\varepsilon)}\right)^\top
\\
&=&
E_N^{(\varepsilon)}
\left(Q_N^{(\varepsilon)}\right)^\top
+
Q_N^{(\varepsilon)}
\left(E_N^{(\varepsilon)}\right)^\top
+
E_N^{(\varepsilon)}
\left(E_N^{(\varepsilon)}\right)^\top.
\end{eqnarray*}
Then, using~\eqref{estimates1} and~\eqref{estimates2}, there exist
constants $c_\varepsilon,C_\varepsilon>0$ such that, for $N$ sufficiently large,
\[
\left\|E_N^{(\varepsilon)}\right\|_F
\le
\frac{c_\varepsilon}{\sqrt N},
\qquad
\left\|Q_N^{(\varepsilon)}\right\|_F
\le
C_\varepsilon\sqrt N.
\]
Using the triangle inequality and the inequality
\[
\|AB\|_F\le\|A\|_F\|B\|_F,
\]
we obtain
\begin{equation*}
\left\|
H_N^{(\varepsilon)}
\left(H_N^{(\varepsilon)}\right)^\top
-
Q_N^{(\varepsilon)}
\left(Q_N^{(\varepsilon)}\right)^\top
\right\|_F
\le
2
\left\|E_N^{(\varepsilon)}\right\|_F
\left\|Q_N^{(\varepsilon)}\right\|_F
+
\left\|E_N^{(\varepsilon)}\right\|_F^2
\le
2c_\varepsilon C_\varepsilon
+
\frac{c_\varepsilon^2}{N}.
\end{equation*}
Therefore,
\[
\frac{1}{N}
\left\|
H_N^{(\varepsilon)}
\left(H_N^{(\varepsilon)}\right)^\top
-
Q_N^{(\varepsilon)}
\left(Q_N^{(\varepsilon)}\right)^\top
\right\|_F^2
\le
\frac{1}{N}
\left(
2c_\varepsilon C_\varepsilon
+
\frac{c_\varepsilon^2}{N}
\right)^2
\longrightarrow0,
\quad N\to\infty.
\]
Hence, the sequence of differences is zero-distributed. Therefore, by~\eqref{eq:Qeps-glt} and the GLT properties, we have
\begin{equation}\label{eq:Heps-gram}
\left\{
H_N^{(\varepsilon)}
\left(H_N^{(\varepsilon)}\right)^\top
\right\}_N
\sim_{\mathrm{GLT}}
\kappa_\varepsilon.
\end{equation}
We next compare the truncated Gram matrices with $H_NH_N^\top$. To this end, adding and subtracting $H_N^{(\varepsilon)}H_N^\top$, we obtain
\begin{eqnarray*}
H_NH_N^\top
-
H_N^{(\varepsilon)}
\left(H_N^{(\varepsilon)}\right)^\top
&=&
H_NH_N^\top
-
H_N^{(\varepsilon)}H_N^\top
+
H_N^{(\varepsilon)}H_N^\top
-
H_N^{(\varepsilon)}
\left(H_N^{(\varepsilon)}\right)^\top
\\
&=&
\left(H_N-H_N^{(\varepsilon)}\right)H_N^\top
+
H_N^{(\varepsilon)}
\left(H_N-H_N^{(\varepsilon)}\right)^\top.
\end{eqnarray*}
Then, we get
\begin{eqnarray*}
\operatorname{rank}\left(
H_NH_N^\top
-
H_N^{(\varepsilon)}
\left(H_N^{(\varepsilon)}\right)^\top
\right)
&\le&
\operatorname{rank}\left(
\left(H_N-H_N^{(\varepsilon)}\right)H_N^\top
\right)+
\operatorname{rank}\left(
H_N^{(\varepsilon)}
\left(H_N-H_N^{(\varepsilon)}\right)^\top
\right)
\\
&\le&
2\operatorname{rank}\left(
H_N-H_N^{(\varepsilon)}
\right).
\end{eqnarray*}
Hence, by~\eqref{eq:rank-cut}, we have
\begin{equation}\label{eq:gram-rank-cut}
\operatorname{rank}\left(
H_NH_N^\top
-
H_N^{(\varepsilon)}
\left(H_N^{(\varepsilon)}\right)^\top
\right)
\le
6\left\lceil\varepsilon N\right\rceil.
\end{equation}
Let $\left\{\varepsilon_m\right\}_m$ be a positive sequence such that
\[
\varepsilon_m\longrightarrow0,
\quad m\to\infty.
\]
From~\eqref{eq:gram-rank-cut}, for $N$ sufficiently large, we have
\[
\operatorname{rank}\left(
H_NH_N^\top
-
H_N^{(\varepsilon_m)}
\left(H_N^{(\varepsilon_m)}\right)^\top
\right)
\le
6\left\lceil\varepsilon_m N\right\rceil
\le
12\varepsilon_m N.
\]
Since
\[
12\varepsilon_m\longrightarrow0,
\quad m\to\infty,
\]
it follows from the definition of an approximating class of sequences~\cite[Definition~1]{Garoni:2017:TFA} that
\[
\left\{
H_N^{(\varepsilon_m)}
\left(H_N^{(\varepsilon_m)}\right)^\top
\right\}_N
\]
is an a.c.s. for
\[
\left\{H_NH_N^\top\right\}_N.
\]
Moreover, since  by~\eqref{eq:kappa-eps-limit},
\[
\kappa_{\varepsilon_m}\to\kappa_{\alpha,\beta,g},
\quad \text{almost everywhere},
\]
and hence in measure on $(0,1)\times[-\pi,\pi]$, using~\eqref{eq:Heps-gram} and the a.c.s. closure property of the GLT class, we obtain
\[
\left\{H_NH_N^\top\right\}_N
\sim_{\mathrm{GLT}}
\kappa_{\alpha,\beta,g},
\]
which proves~\eqref{eq:gram-glt-main}.

It remains to prove~\eqref{eq:main-sv}. Since $H_NH_N^\top$ is symmetric and positive semidefinite, by~\eqref{eq:gram-glt-main} and GLT~6 in~\cite{Guessab:2026:SJW}, applied to the continuous function
\[
f(z)=\sqrt{|z|},
\]
we obtain
\[
\left\{\left(H_NH_N^\top\right)^{1/2}\right\}_N
\sim_{\mathrm{GLT}}
\sqrt{\kappa_{\alpha,\beta,g}},
\]
where we have used that $\kappa_{\alpha,\beta,g}\ge0$ almost everywhere. Hence, by GLT~1 in~\cite{Guessab:2026:SJW}, for any $F\in C_c(\RR)$,
\begin{equation}\label{amsanananasasa}
    \lim_{N\to\infty}
\frac{1}{N}
\sum_{j=1}^N
F\left(
\lambda_j\left(
\left(H_NH_N^\top\right)^{1/2}
\right)
\right)
=
\frac{1}{2\pi}
\int_0^1
\int_{-\pi}^{\pi}
F\left(
\sqrt{\kappa_{\alpha,\beta,g}(y,\theta)}
\right)
\dd\theta\dd y.
\end{equation}
Moreover, since $\left(H_NH_N^\top\right)^{1/2}$ is symmetric and positive semidefinite, we obviously have
$\left[ \lambda_j\left(\left(H_NH_N^\top\right)^{1/2}\right)
\right]^2
=
\left[\sigma_j\left(H_N\right)\right]^2$
and since singular values are nonnegative, it follows that
\[
\lambda_j\left(\left(H_NH_N^\top\right)^{1/2}\right)
=
\sigma_j\left(H_N\right),
\quad j=1,\ldots,N.
\]
Hence, by~\eqref{amsanananasasa}, we have
\[
\lim_{N\to\infty}
\frac{1}{N}
\sum_{j=1}^N
F\left(\sigma_j\left(H_N\right)\right)
=
\frac{1}{2\pi}
\int_0^1
\int_{-\pi}^{\pi}
F\left(
\sqrt{\kappa_{\alpha,\beta,g}(y,\theta)}
\right)
\dd\theta\dd y.
\]
By the definition of singular value distribution, we infer the relation
\[
\left\{H_N\right\}_N
\sim_\sigma
\sqrt{\kappa_{\alpha,\beta,g}},
\]
which proves~\eqref{eq:main-sv}.
\end{proof}

\section{Further results}\label{sec:further}

The symbol obtained in the previous section contains information which cannot be recovered from the zero-distribution results for scaled matrices. For the uniform mesh~\eqref{eq:nodes}, corresponding to the choice
\begin{equation}\label{eq:uniform-g}
g(y)=2y-1,
\end{equation}
we have
\begin{equation}\label{eq:q-uniform}
q(y)=\frac{1}{\sqrt{y(1-y)}}.
\end{equation}
The following result is a direct consequence of Corollary~\ref{cor:zero-set}. 

\begin{corollary}\label{cor:uniform-zero}
For the uniform mesh corresponding to~\eqref{eq:uniform-g}, the normalized measure of the zero set is
\begin{equation}\label{eq:uniform-zero}
\mu_0
=
\sqrt{1-\frac{4}{\pi^2}}
-\frac{2}{\pi}
\arcsin\left(\sqrt{1-\frac{4}{\pi^2}}\right)
\simeq 0.21051366.
\end{equation}
\end{corollary}

\begin{proof}
By~\eqref{eq:zero-mu} and~\eqref{eq:q-uniform},
\[
\mu_0
=
\int_0^1
\left(1-\frac{1}{\pi\sqrt{y(1-y)}}\right)_+\dd y.
\]
The function under the integral is positive exactly when
\[
y(1-y)>\frac{1}{\pi^2},
\]
that is,
\[
s_{-}
<
y
<
s_{+},
\]
where
\[
s_{-}=\frac{1-\sqrt{1-\frac{4}{\pi^2}}}{2}, \quad s_{+}=\frac{1+\sqrt{1-\frac{4}{\pi^2}}}{2}.
\]
Hence,
\begin{equation*}
\mu_0
=
\int_{s_{-}}^
{s_{+}}
\left(1-\frac{1}{\pi\sqrt{y(1-y)}}\right)\dd y
=
\sqrt{1-\frac{4}{\pi^2}}
-\frac{2}{\pi}
\arcsin\left(\sqrt{1-\frac{4}{\pi^2}}\right),
\end{equation*}
which gives~\eqref{eq:uniform-zero}.
\end{proof}

\begin{remark}\label{rem:mesh}
By Corollary~\ref{cor:zero-set}, the zero set of the symbol is completely determined by the function $q$, and therefore by the mesh. The Jacobi parameters $\alpha$ and $\beta$ appear only in the strictly positive factor in~\eqref{eq:kappa-main}. Hence, they affect the values of the symbol on its support.
\end{remark}

\begin{remark}
For the four classical Chebyshev weights, the factor multiplying the
summation in~\eqref{eq:kappa-main} can be written in a simple form.
Indeed, using the half-angle identities
\[
\left[\sin\left(\frac{\vartheta(y)}{2}\right)\right]^2
=
\frac{1-g(y)}{2},
\quad
\left[\cos\left(\frac{\vartheta(y)}{2}\right)\right]^2
=
\frac{1+g(y)}{2},
\]
we obtain
\[
2^{2\alpha+2\beta-1}
\left(\sin\left(\frac{\vartheta(y)}{2}\right)\right)^{2\alpha-3}
\left(\cos\left(\frac{\vartheta(y)}{2}\right)\right)^{2\beta-3}
=
\begin{cases}
\dfrac{2}{\left(1-(g(y))^2\right)^2},
& \alpha=\beta=-\dfrac12,\\[2ex]
\dfrac{8}{1-(g(y))^2},
& \alpha=\beta=\dfrac12,\\[2ex]
\dfrac{4}{(1-g(y))\left(1-(g(y))^2\right)},
& \alpha=-\dfrac12,\ \beta=\dfrac12,\\[2ex]
\dfrac{4}{(1+g(y))\left(1-(g(y))^2\right)},
& \alpha=\dfrac12,\ \beta=-\dfrac12.
\end{cases}
\]
These four cases correspond, respectively, to the Chebyshev weights
of the first, second, third, and fourth kind.
\end{remark}

\section{Interpretation of the theory and numerical results}\label{sec:numerical}

In this section, we discuss the consequences of the main GLT result
and present numerical experiments concerning the singular
value distribution and conditioning of the histopolation
matrices.

\subsection{Distribution of the histopolation matrices}

We present numerical experiments supporting the singular value distribution obtained in Theorem~\ref{thm:main}. More precisely, we compare the singular values of the unscaled histopolation matrices $H_N$ with the corresponding samples of the singular value symbol
\[
\sqrt{\kappa_{\alpha,\beta,g}(y,\theta)}.
\]
The singular values of $H_N$ are sorted in decreasing order and plotted against the normalized index. For the symbol, we sample $\sqrt{\kappa_{\alpha,\beta,g}(y,\theta)}$ on a uniform tensor grid in $(y,\theta)$, sort the resulting values in decreasing order, and plot the corresponding nonincreasing rearrangement on $[0,1]$, see~\cite{Barbarino:2022:CAM}. All computations are performed in MATLAB. The entries of the histopolation matrices are computed by Gauss--Legendre quadrature with $105$ quadrature points on each cell. The singular values of the resulting dense matrices are computed directly using the MATLAB command \texttt{svd}. We notice that the smallest computed singular values may be affected
by roundoff errors and do not necessarily provide accurate relative
approximations of the corresponding exact values. Therefore, the final
part of the spectrum is intended only as a qualitative indication of
the decay. For the approximation of the symbol, we use a uniform tensor grid with $800$ points in the variable $y$ and $1600$ points in the variable $\theta$.

We consider the symmetric case
\[
\alpha=\beta=2
\]
and the matrix sizes
\[
N=500,\quad N=1000,\quad N=2000.
\]
The nodes are chosen as in~\eqref{eq:nodes} using the three functions
\begin{equation}\label{eq:num-meshes}
g_1(y)=-1+2\frac{e^y-1}{e-1},
\qquad
g_2(y)=-1+2y^2,
\qquad
g_3(y)=-\cos(\pi y).
\end{equation}
The first two are the same families of nonuniform nodes considered in
the numerical experiments of~\cite{Guessab:2026:SJW}, while $g_3$
is the cosine mesh for which $\mu_0=0$.
The results are shown in Figure~\ref{fig:HN-symbol-numerical}.
For $g_1$ and $g_2$, the singular values of $H_N$ closely follow the
sorted samples of $\sqrt{\kappa_{\alpha,\beta,g}}$ and exhibit a sharp
decay in the final part of the spectrum. For $g_3$, the singular values
again follow the symbol, but no sharp decay occurs before the right
endpoint, since the zero set of the symbol has measure zero.

\begin{figure}[t]
\centering
\includegraphics[width=0.49\textwidth]{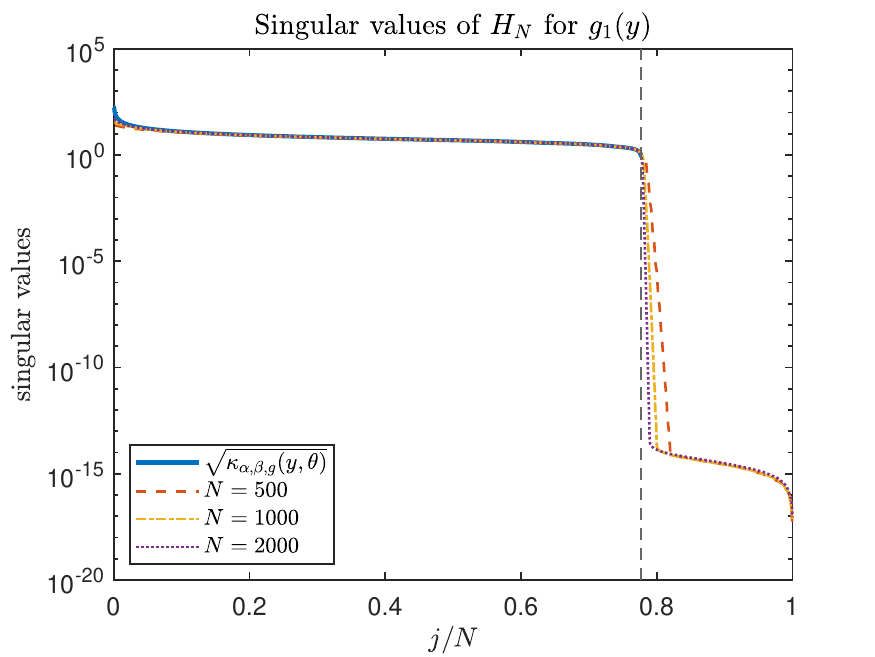}
\includegraphics[width=0.49\textwidth]{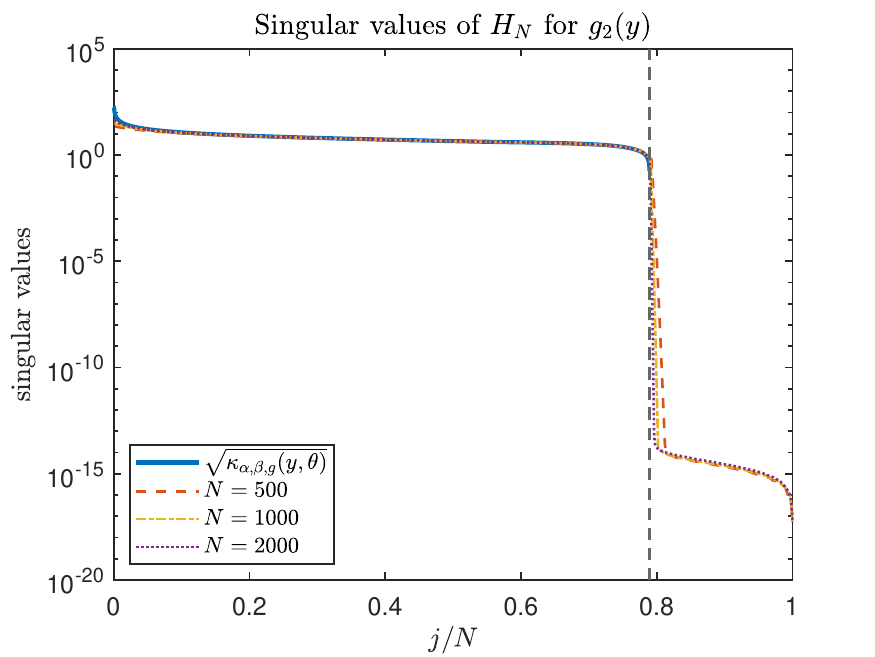}
\includegraphics[width=0.49\textwidth]{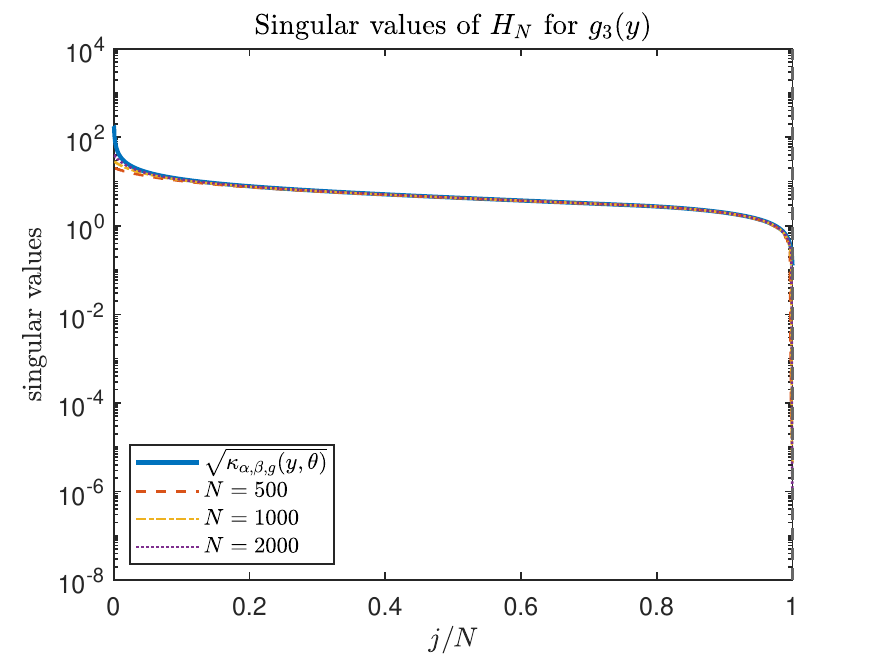}
\caption{Singular values of $H_N$ for $N=500,1000,2000$ compared with the
sorted samples of $\sqrt{\kappa_{\alpha,\beta,g}}$, with $\alpha=\beta=2$.
The nodes are generated by the mesh functions
$g_1(y)=-1+2(e^y-1)/(e-1)$ (top left),
$g_2(y)=-1+2y^2$ (top right), and
$g_3(y)=-\cos(\pi y)$ (bottom). }
\label{fig:HN-symbol-numerical}
\end{figure}
We denote by $\rho_{g_i}$, $i=1,2,3$, the normalized measure of the set
where the symbol is nonzero for the mesh generated by $g_i$. More
generally, by Corollary~\ref{cor:zero-set},
\begin{equation}\label{eq:rho-numerical}
\rho_{g_i}
=
1-\mu_0
=
\int_0^1
\min\left\{1,\frac{q(y)}{\pi}\right\}\dd y.
\end{equation}
For the exponential mesh $g_1$, numerical evaluation
of~\eqref{eq:rho-numerical} gives
\[
\rho_{g_1}\simeq0.776875,
\]
while for the quadratic mesh $g_2$ we obtain
\[
\rho_{g_2}\simeq0.789486.
\]
The vertical dashed lines in the first two plots of
Figure~\ref{fig:HN-symbol-numerical} indicate the corresponding values
of $\rho_{g_1}$ and $\rho_{g_2}$. In both cases, the sharp decay of the
singular values occurs near these points. For the cosine mesh $g_3$,
we have $q(y)=\pi$ for $y\in(0,1)$, and therefore
\[
\rho_{g_3}=1.
\]

Overall, the numerical results support the singular value distribution
obtained in Theorem~\ref{thm:main}. For $g_1$ and $g_2$, the sharp decay
of the singular values begins near the values $\rho_{g_1}$ and
$\rho_{g_2}$ predicted by Corollary~\ref{cor:zero-set}. For $g_3$, we
have $\rho_{g_3}=1$, and no analogous decay occurs inside the interval.

\subsection{Conditioning of the histopolation matrices}

As discussed above, when $\mu_0>0$, the singular value symbol vanishes
on a set of positive measure. Hence, as the matrix size increases, an
increasing number of singular values becomes very small. Since $H_N$ is
nonsingular for any $N$, this suggests a  deterioration of the
conditioning. A similar behavior is known in the Toeplitz setting, where a symbol
vanishing on a set of positive measure or having isolated zeros of exponential order necessarily leads to exponentially small
extreme eigenvalues and, consequently, to exponential growth of the
spectral condition number; see, e.g.,
\cite{Capizzano:2000:HBC,Tilli:2003:UBO}. This observation motivates us
to investigate numerically whether a similar exponential trend occurs
for the histopolation matrices. We point out that
Theorem~\ref{thm:main} gives the singular value distribution, but does
not provide the asymptotic behavior of the smallest singular value. We therefore consider the spectral condition number of the unscaled
histopolation matrices, namely
\[
\kappa_2\left(H_N\right)
=
\frac{\sigma_{\max}\left(H_N\right)}
{\sigma_{\min}\left(H_N\right)}.
\]
In analogy with the previous numerical experiment, we consider the
symmetric case $\alpha=\beta=2$ and the three meshes $g_1$, $g_2$, and
$g_3$ defined in~\eqref{eq:num-meshes}. For increasing values of $N$, we compute $\kappa_2\left(H_N\right)$ for all three meshes. Since the smallest singular values become very small as $N$
increases, we restrict the computation to the range of matrix
sizes shown in Figure~\ref{fig:conditioning}. For the values of $N$ considered, the numerical results for $g_1$ and
$g_2$ suggest an exponential growth of the form
\[
\kappa_2\left(H_N\right) \simeq C a^N,
\quad C>0,\quad a>1.
\]
To estimate the parameter $a$, we perform a least squares approximation
in logarithmic scale, that is,
\[
\log \kappa_2\left(H_N\right)
\simeq
\log C + N\log a.
\]
For the exponential mesh $g_1$, we obtain
\[
a_1 \simeq 2.0005,
\]
while for the quadratic mesh $g_2$ we get
\[
a_2 \simeq 3.9274.
\]
In both cases, the computed condition numbers follow closely the
corresponding exponential trend. These numerical experiments therefore indicate a rapid deterioration of the conditioning of the histopolation matrices as $N$ increases.

\begin{figure}[ht]
\centering
\includegraphics[width=0.49\textwidth]{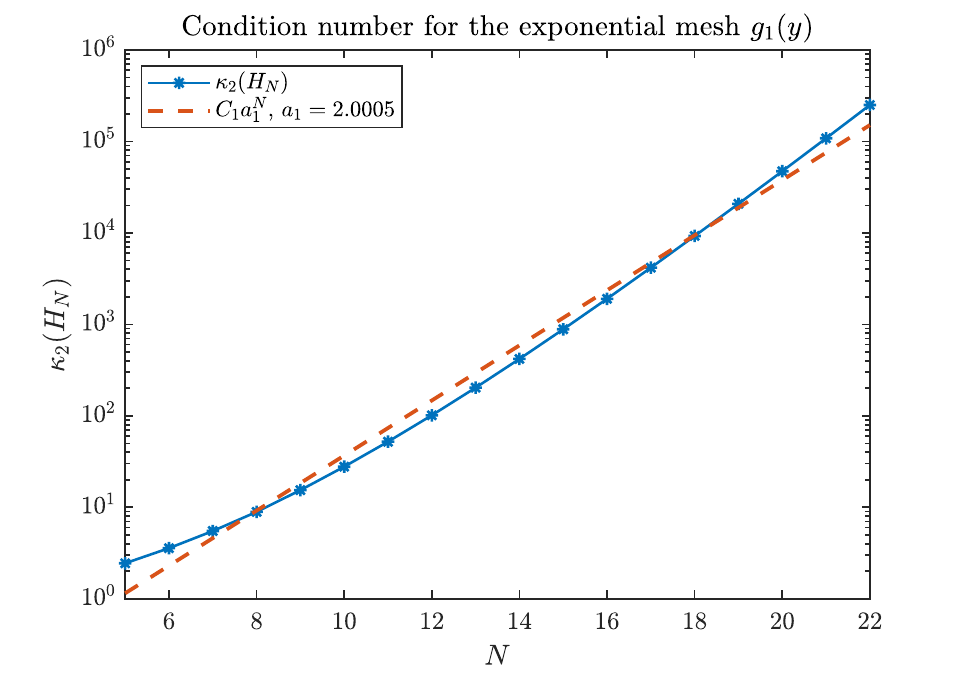}
\hfill
\includegraphics[width=0.49\textwidth]{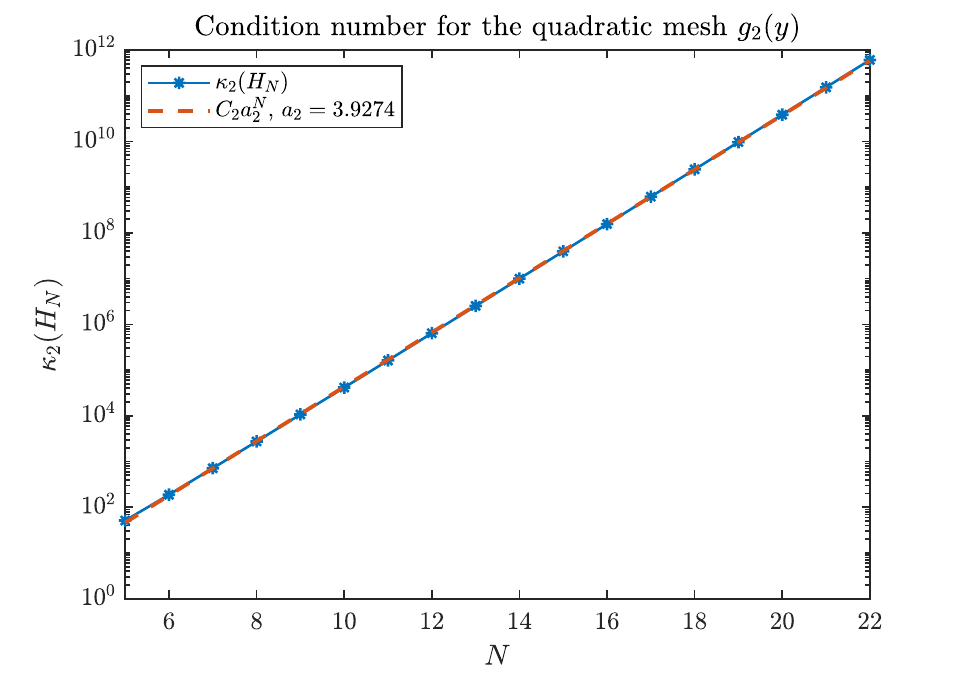}
\includegraphics[width=0.49\textwidth]{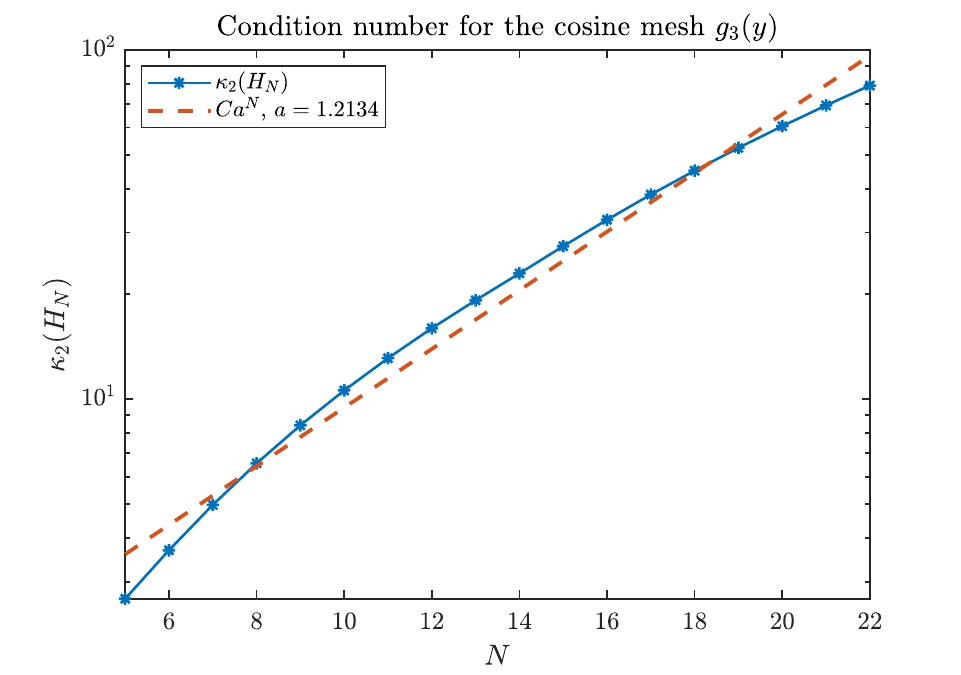}
\caption{Spectral condition number $\kappa_2\left(H_N\right)$ and corresponding exponential approximations $C_i a_i^N$, $i=1,2,3$, for $\alpha=\beta=2$.}
\label{fig:conditioning}
\end{figure}

For the cosine mesh $g_3$, for which $\mu_0=0$, the condition number
grows substantially more slowly than for $g_1$ and $g_2$. Moreover,
over the range of $N$ considered, the exponential model used for the
first two meshes provides a less accurate fit, while the numerical
data suggest instead an algebraic growth of the condition number, in line again with the purely Toeplitz setting \cite{Serra:1998:OTE,Bottcher:1998:OTC}.

\section{Conclusions}\label{sec:conclusions}

In this paper, we have determined the singular value distribution of the unscaled Jacobi weighted histopolation matrices. Starting from the classical asymptotic formula for Jacobi polynomials, we have derived an asymptotic representation of the weighted cell averages. This representation allowed us to analyze the Gram sequence associated with the histopolation matrices by means of GLT theory and to identify its symbol. As a consequence, we have obtained the singular value symbol of the histopolation matrices. The explicit form of the symbol also allows us to characterize its zero set. We have shown that this set depends only on the mesh, while the Jacobi parameters $\alpha$ and $\beta$ affect the values of the symbol on its support. The numerical experiments support the theoretical results and reproduce the sharp decay observed in the final part of the singular value spectrum. The normalized measure $\mu_0$ of the zero set also provides a direct
interpretation of the numerical results. Indeed, $1-\mu_0$ is the
normalized measure of the set where the singular value symbol is
nonzero, and the sharp decay observed for $g_1$ and $g_2$ begins
near this normalized index. More generally, when $\mu_0>0$,
Theorem~\ref{thm:main} implies the discrete asymptotic ill-posedness
of the histopolation problem. Since $\mu_0$ depends only on the mesh,
this conclusion holds for all Jacobi parameters $\alpha,\beta>-1$.

Finally, we think that the tools developed in the Toeplitz setting \cite{Serra:1998:OTE,Bottcher:1998:OTC,Capizzano:2000:HBC,Tilli:2003:UBO} could be the key for proving formally the exponential/algebraic growth of the conditioning, depending on the parameter $\mu_0$ and on the order of the zeros of the GLT symbol.

\section*{Declarations}

\noindent
\textbf{Corresponding author}\\
Federico Nudo, email federico.nudo@unical.it\\

\noindent
\textbf{Conflict of Interest}\\
The authors declare that they have no conflict of interest.\\

\noindent
\textbf{Funding statement}\\
This research was supported by the GNCS-INdAM 2026 projects
\emph{``Metodi polinomiali e kernel per l'approssimazione da dati discreti e integrali con software OS''} and \emph{``Metodi strutturati per il signal processing avanzato''}
(CUP E53C25002010001).
The work of F. Nudo was funded by the European Union -- NextGenerationEU under the Italian National Recovery and Resilience Plan (PNRR), Mission 4, Component 2, Investment 1.2
\lq\lq Finanziamento di progetti presentati da giovani ricercatori\rq\rq,
pursuant to MUR Decree No.~47/2025.\\

\noindent
\textbf{Author Contributions}\\
Allal Guessab, Federico Nudo and Stefano Serra-Capizzano contributed equally to the conception, development, and writing of this manuscript.
All authors have read and approved the final version of the paper. For this reason, the order of authorship is alphabetical.\\

\noindent
\textbf{Acknowledgement}\\
This research was carried out as part of RITA \textquotedblleft Research ITalian network on Approximation'' and as part of the UMI group \enquote{Teoria dell'Approssimazione e Applicazioni}.\\

\noindent
\textbf{Data Availability}\\
No external datasets were used in this study. The numerical data reported
in the paper were generated computationally as part of the numerical
experiments.

\bibliographystyle{elsarticle-num}
\bibliography{Jacobi_histopolation_symbol}

\end{document}